\documentclass[11pt]{amsart}

\usepackage{amsmath, amsthm, amssymb}

\calclayout

\usepackage{graphicx}
\usepackage{tikz-cd}
\usepackage{tikz}
\usetikzlibrary{positioning}
\usetikzlibrary{hobby}
\usepackage{mathtools}

\newcommand{\N}{\mathbb{N}}
\newcommand{\Q}{\mathbb{Q}}
\newcommand{\R}{\mathbb{R}}

\DeclareMathOperator*{\diam}{diam}
\DeclareMathOperator{\vol}{Vol}

\newtheorem{thm}{Theorem}[section]

\newtheorem{ex}[thm]{Example}
\newtheorem{cor}[thm]{Corollary}   
\newtheorem{lem}[thm]{Lemma}   

\newtheorem{defn}[thm]{Definition}
\newtheorem{rmrk}[thm]{Remark}

\begin{document}

\title[Convergence for Sequences of Bounded Rough Riemannian Metrics]{Lipschitz Bounds and Uniform Convergence for Sequences of Bounded Rough Riemannian Metrics}

\author[B. Allen]{ Brian Allen }
\author[B. Falcao]{ Bernardo Falcao }
\author[H. Pacheco]{ Harry Pacheco }
\author[B. Sanchez]{ Bryan Sanchez }

\maketitle

\begin{abstract}
    Here we study what we call bounded rough Riemannian metrics $(M,g)$, which are positive definite, symmetric tensors on each tangent space, $T_pM$, which are bounded and measurable as functions in coordinates. This is enough structure to study the length space given by taking the infimum of the length of all piecewise smooth curves connecting points $p,q \in M$. The goal is to find the weakest conditions one can place on $g$ which can guarantee Lipschitz or uniform bounds from above and below. For each condition, an example is given showing that the condition cannot be weakened any further which also explores the geometric intuition.
\end{abstract}

\section{Introduction}

Here we study what we call bounded rough Riemannian metrics $(M,g)$, which are positive definite, symmetric tensors on each tangent space, $T_pM$, which are bounded above, uniformly bounded away from zero, and measurable as functions in coordinates. A definition of rough Riemannian metrics was given by L. Bandara \cite{BandaraRough} and bounded rough Riemannian metrics are a subclass of the metrics he defines and studies. The definition of bounded rough Riemannian metric provides enough structure to study the length space associate to $(M,g)$ given by the length of piecewise smooth curves $\gamma:[0,1]\rightarrow M$ 
\begin{align}\label{def-RoughRiemannianLengthIntro}
 L_g(\gamma)&= \int_0^1  \sqrt{g(\gamma',\gamma')}dt,
\end{align}
and the distance between points $p,q \in M$ by minimizing the lengths of curves connecting $p$ and $q$
\begin{align}\label{def-RoughRiemannianDistanceIntro}
    d_g(p,q)=\inf\{L_g(\gamma): \gamma \text{ piecewise smooth}, \gamma(0)=p,\gamma(1)=q\}.
\end{align}
The goal is to find the weakest conditions one can place on $g$ which can guarantee Lipschitz or uniform bounds from above and below. This study is motivated by wanting to understand the metric convergence of sequences of smooth Riemannian metrics under scalar curvature bounds where a priori limits could be length spaces or metric spaces which are not even continuous Riemannian manifolds.

One important conjecture we would like to understand better is the scalar curvature compactness conjecture. We can give an imprecise version of this conjecture by asking, What additional conditions are necessary on a sequence of Riemannian manifolds with positive scalar curvature so that a subsequence converges in the Sormani-Wenger Intrinsic Flat sense to a metric space with a notion of positive scalar curvature? One can find a precise version of this conjecture stated by C. Sormani  in \cite{IAS}. J. Park, W. Tian, and C. Wang \cite{PTW} study rotationally symmetric warped products,   C. Sormani, W. Tian, and C. Wang  considered $\mathbb{S}^1\times_f \mathbb{S}^2$ warped products with an important example already explored \cite{WCS} and $W^{1,p}$, $p < 2$ convergence obtained \cite{tian2023compactness}. B. Allen, W. Tian, and C. Wang \cite{allen-tian-wang} investigate metrics conformal to the round sphere and establish a similar convergence as Dong \cite{Dong-PMT_Stability} for the stability of the positive mass theorem. Due to the fact that lower bounds on scalar curvature tend to lead to $W^{1,p}$ compactness where the limiting function could be thought of as a $W^{1,p}$ Riemannian metric, it is natural to investigate much weaker notions of Riemannian metrics and their metric space convergence properties.

The first named author and E. Bryden \cite{Allen-Bryden-First, Allen-Bryden-Second} have previously studied $W^{1,p}$ and $L^p$ bounds on Riemannian metrics, showing that various bounds on the corresponding distance functions can be obtained even under these weak assumptions. C. Dong, Y. Li, and K. Xu \cite{Sobolev-Trace-Convergence} also studied $W^{1,p}$ bounds on Riemannian metrics and their convergence properties as well as C. Aldana, G. Carron, and S. Tapie \cite{Aldana-Carron-Tapie} in the case of conformal Riemannian metrics. As previously mentioned, the definition of rough Riemannian metrics has been stated by L. Bandara \cite{BandaraRough}, and additionally studied by  L. Bandara and P. Bryan \cite{BandaraBryan}, L. Bandara, M. Nursultanov, and J. Rowlett \cite{BNR}, and by L. Bandara and A. Hassan \cite{BandaraHassan}. Since we are concerned with understanding the length space structure of smooth Riemannian metrics, it makes sense to add boundedness to the definition of rough Riemannian metrics in our case so that we are always guaranteed a distance function defined by minimizing the lengths of curves. Then all of our results for bounded rough Riemannian metrics apply to sequences of smooth Riemannian metrics and their limits. Continuing the more general study, L. Badara and A. Hassan \cite{BandaraHassan}  investigate metric and length space properties of the entire space of rough Riemannian metrics on smooth and connected manifolds. In a series of papers, G. De Cecco and G. Palmieri \cite{DeCeccoPalmieri, DeCeccoPalmieri2, DeCeccoPalmieri3, DeCeccoPalmieri4} study Lipschitz Riemannian manfiolds and LIP Finslerian manifolds, studying several notions of metric spaces defined in these weak context and establishing relationships between these different definitions.  The length space structure of continuous Riemannian metrics was studied by A. Burtscher \cite{BurtscherC0} showing that the induced distance is equivalent to the distance defined as the infimum of the lengths of absolutely continuous curves.

We start with studying Lipschitz bounds from below. Here, the examples of subsection \ref{subsec-Shortcut Examples} show that if a shortcut develops on any set of positive Hausdorff $1$ measure then one should not expect to be able to establish a lower Lipschitz bound. In particular, Theorem \ref{ex-Shortcut on Dissappearing Square} gives an example where the conformal factor is $1$ everywhere except for a rectangle at the center of disappearing side length $\frac{1}{n}$ where the conformal factor is $\frac{1}{n^{\alpha}}, \alpha \in (0,\infty)$. With this in mind we can state our first main result which establishes a lower Lipschitz bound for bounded rough Riemannian metrics.

\begin{thm}\label{thm-Main Theorem Lipschitz Lower Bound}
   Let $g_1$ be a bounded rough Riemannian metric, $g_0$ a Riemannian metric, and $M$ a smooth, connected, compact manifold. If there exists a $U \subset M$ so that $\mathcal{H}_{g_0}^1(M \setminus U) = 0$ and there exists a $c > 0$ so that
    \begin{align}
        g_1(v,v) \ge cg_0(v,v), \quad \forall v \in T_pM, p\in U,\label{eq-LowerBoundOnU}
    \end{align}
    then
    \begin{align}
        d_{g_1}(x,y) \ge c d_{g_0}(x,y).
    \end{align}
\end{thm}

Given the importance of $C^0$ convergence from below of Riemannian metrics identified by the first named author, R. Perales, and C. Sormani \cite{Allen-Perales-Sormani} we state a Corollary of Theorem \ref{thm-Main Theorem Lipschitz Lower Bound} which follows immediately.

\begin{cor}\label{cor-Main Corollary Uniform Lower Bound 1}
      Let $g_n$ be a sequence of bounded rough Riemannian metrics, $g$ a Riemannian metric, and $M$ smooth, connected, and compact. If there exists a set $U \subset M$ so that $\mathcal{H}^1_g(M \setminus U) = 0$ where $C_n\searrow 0$
   \begin{align}
       g_n(v,v) \ge (1-C_n)g(v,v), \quad \forall v \in T_pM,  p \in U,\label{eq-LowerBoundOnU2}
   \end{align}
    then
    \begin{align}
        d_{g_n}(x,y) \ge (1-C_n)d_g(x,y).
    \end{align}
\end{cor}

We now loosen our required goal in order to study uniform control form below on a sequence of bounded rough Riemannian metrics. In this case, examples given by the first named author and C. Sormani \cite{Allen-Sormani, Allen-Sormani-2} have shown that one needs to impose uniform bounds from below on Riemannian metrics in order to find $C^0$ convergence from below of distance functions. Here we are able to weaken this requirement in order to show that we just need the Hausdorff $1$ measure of the set where a shortcut is developing to become small. We test this condition with Theorem \ref{ex-Shortcut Dense Uniform no Lipschitz}, \ref{ex-Shortcut No Uniform No Lipschitz Middle}, and \ref{ex-Shortcut No Uniform No Lipschitz Extreme} where we see that depending on the limiting behavior of the Hausdorff $1$ measure of the shortcut set we can get several different metric spaces in the limit.

\begin{thm}\label{thm-Main Theorem Uniform Lower Bound}
   Let $g_n$ be a sequence of bounded rough Riemannian metrics, $g$ a Riemannian metric, and $M$ smooth, connected, and compact. If there exists a set $U \subset M$ so that $\mathcal{H}^1_g(M \setminus U) \le C_n$ where $C_n\searrow 0$ and a $c >0$ so that
   \begin{align}
       g_n(v,v) \ge c g(v,v), \quad \forall v \in T_pM, p \in U,\label{eq-LowerBoundOnU2}
   \end{align}
    then
    \begin{align}
        d_{g_n}(x,y) \ge cd_g(x,y)-cC_n.
    \end{align}
\end{thm}

Again, given the importance of $C^0$ convergence from below of Riemannian metrics as exemplified by the Volume Above Distance Below theorem of the first named author, R. Perales, and C. Sormani \cite{Allen-Perales-Sormani} we state a  Corollary of Theorem \ref{thm-Main Theorem Uniform Lower Bound} which follows immediately which illustrates how this theorem can be used to show convergence to a particular metric space.

\begin{cor}\label{cor-Main Corollary Uniform Lower Bound 2}
      Let $g_n$ be a sequence of bounded rough Riemannian metrics, $g$ a Riemannian metric, and $M$ smooth, connected, and compact. If there exists a set $U \subset M$ so that $\mathcal{H}^1_g(M \setminus U) \le C_n$ where $C_n\searrow 0$ and
   \begin{align}
       g_n(v,v) \ge (1-C_n)g(v,v), \quad \forall v \in T_pM,  p \in U,\label{eq-LowerBoundOnU2}
   \end{align}
    then
    \begin{align}
        d_{g_n}(x,y) \ge (1-C_n)d_g(x,y)-(1-C_n)C_n.
    \end{align}
\end{cor}

Now we switch gears in order to study upper bounds. In general, it was noticed by the first named author and C. Sormani \cite{Allen-Sormani, Allen-Sormani-2} that one should expect to require much weaker hypotheses in order to establish upper bounds on the distance induced by Riemannian metrics. This is due in part to the fact that the distance function is defined as the infimum of lengths, making upper bounds generally easier to obtain. We start by establishing an upper Lipschitz bound where we allow the sequence to blow up at any rate as long as it occurs on a set of volume zero. One can think of this as a type of $L^{\infty}$ bound on the sequence. The first named author already established H\"{o}lder bounds, $d_{g_n}\le C d^{\frac{p-n}{n}}$ on sequences of $C^0$ Riemannian manifolds where $\|g_n\|_{L^{\frac{p}{2}}_{g}(M)}\le C, p>n$ is observed to be the correct hypothesis. This is seen as analogous to a Morrey inequality where we think of $|g_n|_g$ as an upper bound of the gradient of $d_{g_n}$ with respect to $g$. As we consider allowing $p \rightarrow n$ it is the $L^{\frac{p}{2}}_{g}(M)$ formally becomes an $L^{\infty}$ bound and the H\"{o}lder bound becomes a Lipschitz bound. The following theorem confirms that this formal argument is correct. In addition, the examples in subsection \ref{subsec-Blowup Examples} confirm that this hypothesis is necessary.

\begin{thm}\label{thm-Main Theorem Lipschitz Upper Bound}
    Let $g_n$ be a sequence of bounded rough Riemannian metrics, $g$ a  Riemannian metric, and $M$ smooth, connected, and compact. If there exists an open set $U \subset M$ so that $\vol_g(M \setminus U) = 0$ and there exists a $C > 0$ so that
    \begin{align}
        g_n(v,v) \le C g(v,v), \quad \forall v\in T_pM, v \in U,\label{eq-LowerBoundOnU}
    \end{align}
    then
    \begin{align}
        d_{g_n}(x,y) \le  C d_g(x,y).
    \end{align}
    
\end{thm}

For our last main theorem we turn to establishing uniform control on a sequence of distance functions from above. Here we are able to establish a theorem which we know not to be optimal by Theorem \ref{ex-Blowup Uniform no Lipschitz}. That being said, it allows for the sequence to blow up and gives a hypothesis which is geometrically natural. A similar result was obtained by the first and fourth named authors and Y. Torres \cite{BST} for warped product bounded rough Riemannian metrics. 

\begin{thm}\label{thm-Main Theorem Uniform Upper Bound}
  Let $g_n$ be a sequence of bounded rough Riemannian metrics, $g$ a Riemannian metric, and $M$ smooth, connected, and compact. If there exists a sequence of sets $U_n \subset M$ so that $\displaystyle M \setminus U_n$ can be decomposed into connected components $W_1,W_2,...$ where $\displaystyle\sum _{k=1}^{\infty} \diam(W_k) \le V_n $, $V_n \searrow 0$ as $n \rightarrow 0$ and there exists a $C > 0$ so that
    \begin{align}
        g_n(v,v) &\le Cg(v,v), \quad \forall v \in T_pM, p \in U_n,\label{eq-LowerBoundOnU}
        \\V_ng_n(v,v) & \le C_ng(v,v),\quad \forall v \in T_pM, p \in M \setminus U_n,\label{eq-LowerBoundOnUComplement}
    \end{align}
      where $C_n \searrow 0$ as $n \rightarrow \infty$, then
    \begin{align}
        d_{g_n}(x,y) \le  C d_g(x,y)+C_n.
    \end{align}
  
\end{thm}

In section \ref{sec-Background}, we review necessary background and fix notation which will be used throughout the rest of the paper.

In section \ref{sec-Examples}, we state and prove many examples which serve to illustrate the importance of the assumptions of the main theorems.

In section \ref{sec-Main Proofs}, we give the proof of the main theorems. We note that the proofs are fairly short which is justified by the necessity of the hypothesis, as illustrated by the examples.

\textbf{Acknowledgment:} We would like to thank the PSC CUNY Traditional B grant and the Lehman College STEM-IN grant for providing summer salary to all the authors on this project. We also thank the Lehman College Math Department, where the second through fourth named authors all studied mathematics as undergraduates, for its continued support. 

\section{Background}\label{sec-Background}
In this section we will review the necessary background material needed to understand the rest of the paper. We start by reviewing a particular class of conformal length spaces which is used when discussing each of the examples in section \ref{sec-Examples}. We then go on to generalize to Bounded Rough Riemannian metrics.

\subsection{Conformal Length Spaces}

By considering any piecewise smooth curve $\gamma:[0,1]\rightarrow [0,1]^2$ we can define a conformal length space on $[0,1]^2 \subset \R^2$ by being given a bounded and measurable function $f:[0,1]^2\rightarrow (0,\infty)$, which distorts the geometry of the plane, resulting in a new length, $L_f(\gamma)$, defined as
\begin{align}\label{def-WeightedLength}
 L_f(\gamma)&= \int_0^1  f(x(t),y(t))\sqrt{x'(t)^2+y'(t)^2}dt. 
\end{align}
The idea is that the length of curves will be distorted by the presence of a weighting factor $f(x(t),y(t))$, called a conformal factor, multiplying the speed. Then one can define a distance between points by minimizing the lengths of curves connecting $p$ and $q$
\begin{align}\label{def-WeightedDistance}
    d_f(p,q)=\inf\{L_f(\gamma): \gamma \text{ piecewise smooth}, \gamma(0)=p,\gamma(1)=q\}.
\end{align}
We notice that the Euclidean distance function is given by $d_1$ which we will denote throughout the paper by $d$. We will also denote the straight line connecting $p,q \in [0,1]^2$ as $\ell_{pq}$ throughout. Given two curves $\gamma, \alpha:[0,1]\rightarrow [0,1]^2$ so that $\gamma$ connects $p$ to $p'$ and $\alpha$ connects $p'$ to $q$, we use the notation $\gamma *\alpha:[0,1]\rightarrow [0,1]^2$ for the concatenation of these two curves.

Often times when studying examples we will be able to deduce a distance bound from a corresponding comparison of the conformal factors which we state and prove now.

\begin{thm}\label{thm-Distance Lower Bound Estimate}
    Show that for any functions $f,g:[0,1]^2\rightarrow (0,\infty)$ if $f(x,y) \ge g(x,y)$ $\forall (x,y) \in [0,1]^2$ then
    \begin{align}
        d_g(p,q) \le d_f(p,q),
    \end{align}
    for any $p,q \in [0,1]^2$.
\end{thm}

\begin{proof}
By the assumption that $f(x,y) \ge g(x,y)$ $\forall (x,y) \in [0,1]^2$ we find
\begin{align}
    \int_0^1 f(x(t), y(t)) \sqrt{x'(t)^2 + y'(t)^2} \, dt 
    \geq 
    \int_0^1 g(x(t), y(t)) \sqrt{x'(t)^2 + y'(t)^2} \, dt,
\end{align}
which can be rewritten as
\begin{align}
L_f(\gamma) \geq L_g(\gamma).
\end{align}

Taking the infimum over all paths \( \gamma \) connecting points \( p \) and \( q \), we obtain
\begin{align}
d_f(p,q) \geq d_g(p,q).    
\end{align}

\end{proof}

\subsection{Bounded Rough Riemannian Metrics}
Since our goal is simple to study the metric space properties of Riemannain metrics we will define a very weak class of metrics for which the metric space structure is defined. One should note that this means that the usual Hopf-Rinow theorem of Riemannian geometry will not be available to us when considering bounded rough Riemannian metrics. 

\begin{defn}\label{def-Rough Riemannian Metric}
 If we let $M$ be a smooth compact manifold then we let $g$ be called a \textbf{bounded rough Riemannian metric} if $g$ is a positive definite symmetric tensor on each tangent space $T_pM$, $p \in M$ so that in coordinates on $M$, $g$ is bounded above, uniformly bounded away from zero, and measurable. 
\end{defn} 
One should notice that we are following the definition given by L. Bandera \cite{BandaraRough} but we are forcing the singular set for the rough Riemannian metric to be empty by imposing the boundedness condition globally instead of just almost everywhere. Then we know that for any piecewise smooth curve $\gamma:[0,1]\rightarrow M$ that $\sqrt{g(\gamma',\gamma')}$ is bounded and measurable and hence $\gamma$ will have a well defined length. Hence we can define the length of $\gamma$ with respect to the bounded rough Riemannian metric $g$ to be
\begin{align}\label{def-RoughRiemannianLength}
 L_g(\gamma)&= \int_0^1  \sqrt{g(\gamma',\gamma')}dt,
\end{align}
and the distance between points $p,q \in M$ by minimizing the lengths of curves connecting $p$ and $q$
\begin{align}\label{def-RoughRiemannianDistance}
    d_g(p,q)=\inf\{L_g(\gamma): \gamma \text{ piecewise smooth}, \gamma(0)=p,\gamma(1)=q\}.
\end{align}
\subsection{Uniform Convergence}

First, we remind the reader of the definition of uniform convergence for a sequence of distance functions.

\begin{defn}\label{defn-Uniform Convergence}
    Let $d_j:X \times X \rightarrow (0,\infty)$ be a sequence of metrics and $d_0:X \times X\rightarrow (0,\infty)$ a metric. Then we say that $d_j \rightarrow d_0$ \textbf{uniformly} if 
    \begin{align}
        \sup_{p,q \in X} \left|d_j(p,q)-d_0(p,q) \right|\rightarrow 0.
    \end{align}
\end{defn}

\begin{rmrk}
    The definition of uniform convergence given in Definition \ref{defn-Uniform Convergence} is equivalent to the definition of convergence for a sequence in the metric space $(\mathcal{C},d_{sup})$ if we let $\mathcal{C}$ be the set of continuous functions from $X \times X \rightarrow (0,\infty)$ and for $f,g \in \mathcal{C}$ we define the distance between points to be 
    \begin{align}
        d_{sup}(f,g)= \sup_{p,q \in X}|f(p,q)-g(p,q)|.
    \end{align}
\end{rmrk}
We now prove a routine squeeze theorem for uniform convergence of distance functions which will be used throughout the paper.

\begin{thm}\label{thm-Squeeze}
     Let $d_j:X \times X \rightarrow (0,\infty)$ be a sequence of metrics and $d_0:X \times X\rightarrow (0,\infty)$ a metric. If for all $p,q \in X$ we find that
     \begin{align}\label{eq-squeeze equation}
         d_0(p,q) - C_j \le d_j(p,q) \le d_0(p,q)+C_j,
     \end{align}
     where $C_j$ is a sequence of real numbers so that $C_j \rightarrow 0$ then $d_j \rightarrow d_0$ uniformly.
\end{thm} 
\begin{proof}
Let $d_j$ be a sequence of metrics and $d_0$ be a metric such that for all $p,q \in X$ we find that 
    \begin{align}\label{eq-squeeze equation}
         d_0(p,q) - C_j \le d_j(p,q) \le d_0(p,q)+C_j .
    \end{align}
Note this is equivalent to $- C_j \le d_j(p,q) - d_0(p,q) \le C_j$, which implies,
    \begin{align}\label{eq-squeeze equation}
          |d_j(p,q) - d_0(p,q)| \le C_j.
    \end{align}
By taking the supremum over all $p,q \in X$ on both sides of \eqref{eq-squeeze equation} we find $\sup |d_j(p,q)-d_0(p,q)| \le C_j$, for any $p,q \in X, j \in \N$. Since $C_j$ is a sequence so that $ C_j\rightarrow0$, we can see
\begin{align}
    \sup_{p,q \in X} |d_j(p,q)-d_0(p,q)| \rightarrow 0.
\end{align}
Therefore, $d_j \rightarrow d_0$ uniformly.
\end{proof}

\subsection{Minimizing Classes of Curves}

In Section \ref{sec-Examples}, we will be studying examples of conformal Bounded Rough Riemannian metrics which are piecewise constant on finitely many regions of $[0,1]^2$. So we now make an observation about minimizing curves on metrics of this type.

\begin{lem}\label{lem-MinimizingCurvesArePiecewiseLinear} 
    Given a conformal factor $f:[0,1]^2 \rightarrow (0,\infty)$ which is  constant on each set $U_1,...,U_n$ so that
    \begin{align}
        \bigcup_{i=1}^nU_i=[0,1]^2, \qquad U_i\cap U_j = \emptyset, i \not = j,
    \end{align}
    and so that $\partial U_i$, $1\le i \le n$ is a rectangle (where we allow lines as degenerate rectangles) composed of at most four straight lines $\ell_i^1,\ell_i^2,\ell_i^3,\ell_i^4$ so that each line is either completely contained in or has trivial intersection with $U_i$. Furthermore, we require $\ell^j_i \in U_i$ if $f$ is smaller on $U_i$ then on the $U_k$ so that $\ell_i^j \in \partial U_k\cap \partial U_i$.
    
    If we let $\mathcal{P}$ be the set of piecewise linear curves which are linear on each rectangle $U_i$ then we find that for $p,q \in [0,1]^2$
    \begin{align}
        d_f(p,q)=\inf\{L_f(\gamma):\gamma:[0,1]\rightarrow [0,1]^2, \gamma \in \mathcal{P}, \gamma(0)=p,\gamma(1)=q\}.
    \end{align}
\end{lem}
\begin{proof}
Let $p,q \in [0,1]^2$ and $\gamma:[0,1]\rightarrow [0,1]^2$ be any piecewise continuous path connecting $p$ to $q$. Let $U_i$ be so that $\gamma([0,1])\cap U_i \not = \emptyset$. If $\gamma$ is completely contained in $U_i$ then we can replace with a straight line completely contained in $U_i$ with shorter length since the conformal factor is constant on $U_i$. 

If $\gamma$ is not completely contained in $U_i$, and only intersects the boundary of $U_i$ once at $p' \in \partial U_i$ then either $p$ or $q$ is in $U_i$. Without loss of generality say $p$, and we can replace the part of the curve connecting $p$ to $p'$ with a straight line curve or a piecewise linear curve that is linear on $U_i$ which will only have less length  than $\gamma$. Hence we can replace the portion of $\gamma$ connecting $p$ to $p'$ with a curve  in the class $\mathcal{P}$. It is important here that we have required linear sides of $U_i$ to be contained in the shorter of the two rectangles it is contained in since otherwise one would have to consider piecewise linear curves whose breaks occurred at the interior of some neighboring rectangle.

If $\gamma$ intersects $\partial U_i$ two or more times then let $p',q'$ be two points where $\gamma$ intersects $\partial U_i$ which occur consecutively in parameterized time of $\gamma$. Since $f$ is constant on $U_i$ then we know that we can replace the portion of $\gamma$ connecting $p'$ to $q'$ with a shorter curve between $p'$ and $q'$ inside $\overline{U}_i$, consisting of either the straight line connecting $p'$ to $q'$ or a piecewise linear curve connecting $p'$ to $q'$ whose non-continuous breaks occur at points in $\partial U_i$.  So we can replace the portion of $\gamma$ connecting $p'$ to $q'$ with a curve from the class $\mathcal{P}$ which is shorter than the portion of $\gamma$ connecting $p'$ to $q'$. By continuing this procedure for each $U_i$ which intersects with $\gamma$ we will replace $\gamma$ by a curve in the class $\mathcal{P}$, which implies the desired conclusion.
\end{proof}

\subsection{Conformal Metric Space Examples}

In this section we investigate two examples of length metric spaces defined by conformal factors. We have chosen these particular examples because they will be the limits of a sequence of bounded rough Riemannian metrics of examples given in section \ref{sec-Examples}.

\begin{ex} 
The degenerate distance function defined with respect to the conformal factor
    \begin{align}
    \hat{f}_{\infty}(x,y)=
    \begin{cases}
       0 & x=\frac{1}{2}, 0 \le y \le 1
        \\ 1 & \text{otherwise}
    \end{cases}
\end{align}
is given by
\begin{align}\label{eq-Quotient Metric Space}
  d_{\hat{f}_\infty}(p,q) = \inf\{|x_0-\frac{1}{2}|+|x_1-\frac{1}{2}|,d(p,q)\}.  
\end{align}
If we define points in $P=\{(\frac{1}{2},y):y \in [0,1]\}$ to be identified as one point, then the induced quotient metric space is given by \eqref{eq-Quotient Metric Space}.
\end{ex}
\begin{proof}

Let $p=(x_0,y_0)$ and $q=(x_1,y_1)$. We will proceed by considering cases.

\textbf{Case 0:} if $x_1,x_0 = \frac{1}{2}$ then $L_{\hat{f}_{\infty}}(\ell_{pq})=0$ and hence $d_{\hat{f}_\infty}(p,q) = 0$.

\textbf{Case 1:} Let $p,q \in \left[ 0, \frac{1}{2} \right] \times[0,1 ]$ consider $\gamma$ to be a path connecting $p$ to $q$ such that $\gamma([0,1]) \cap\{\frac{1}{2}\}\times [0,1] = \emptyset$. We know that by Lemma \ref{lem-MinimizingCurvesArePiecewiseLinear}  that $\ell_{pq}$ is the shortest path amongst all curves $\gamma$ so we have

\begin{align} \label{eq-3.14FirstLowerBound}
   L_{\hat{f}_\infty}(\gamma) \ge L_{\hat{f}_\infty}(\ell_{pq}) = d(p,q).
\end{align}

We can also obtain an upper bound

\begin{align} \label{eq-3.14FirstUpperBound}
    d_{\hat{f}_\infty}(p,q) \le L_{\hat{f}_\infty}(\ell_{pq})=d(p,q)
\end{align}

The other type of path to consider is when $\gamma([0,1]) \cap\{\frac{1}{2}\}\times [0,1] \not = \emptyset$. Denote $p' = (\frac{1}{2},y_0)$ and $q' = (\frac{1}{2}, y_1)$, which will be intersections to the mid line. Let $\gamma_1$ to be the path from $p$ to $p'$, $\gamma_2 $ from $p'$ to $q'$, and $\gamma_3$ from $q'$ to $q$. So we estimate a lower bound

\begin{align} 
    L_{\hat{f}_\infty}(\gamma) &= L_{\hat{f}_\infty}(\gamma_1) + L_{\hat{f}_\infty}(\gamma_2) + L_{\hat{f}_\infty}(\gamma_3)\label{eq-3.14SecondLowerBound}
   \\& \ge L_{\hat{f}_\infty}(\ell_{pp'})+ L_{\hat{f}_\infty}(\ell_{p'q'})+L_{\hat{f}_\infty}(\ell_{q'q})
   \\& = L_{\hat{f}_\infty}(\ell_{pp'})+ L_{\hat{f}_\infty}(\ell_{q'q})
  \\ & = \left|x_0 - \frac{1}{2}\right|+\left|x_1-\frac{1}{2}\right|.\label{eq-3.14SecondLowerBound2}
\end{align}

To obtain an upper bound, we can calculate 

\begin{align} \label{eq-3.14SecondUpperBound}
    d_{\hat{f}_\infty}(p,q) \le L_{\hat{f}_\infty}(\ell_{pp'})+ L_{\hat{f}_\infty}(\ell_{qq'}) = \left|x_0 - \frac{1}{2}\right|+\left|x_1-\frac{1}{2}\right|.
\end{align}

So by taking the infimum over \eqref{eq-3.14FirstLowerBound} and \eqref{eq-3.14SecondLowerBound}-\eqref{eq-3.14SecondLowerBound2} we find that, 
\begin{align}\label{eq-Almostinf}
    L_{\hat{f}_\infty}(\gamma) \ge \inf\{|x_0-\frac{1}{2}|+|x_1-\frac{1}{2}|,d(p,q)\}.
\end{align}

Now by taking the infimum \eqref{eq-Almostinf} we find the estimate
\begin{align}
    d_{\hat{f}_\infty}(p,q) \ge \inf \{|x_0-\frac{1}{2}|+|x_1-\frac{1}{2}|,d(p,q)\}
\end{align}

\textbf{Case 2:} Let $p,q \in \left[  \frac{1}{2},1 \right] \times \left[0,1 \right] $ should be symmetric to Case 1. 

\textbf{Case 3:}  Let $ p \in \left[ 0, \frac{1}{2} \right] \times \left[0,1 \right]$  and $q\in \left[ \frac{1}{2} ,1\right] \times [0,1]$ by construction of the path it will always be the case that $\gamma([0,1]) \cap\{\frac{1}{2}\}\times [0,1] \not = \emptyset$. Notice since the hypotenuse is always larger than the side of a triangle, we have that $\inf\{|x_0-\frac{1}{2}|+|x_1-\frac{1}{2}|, d(p,q)\}=|x_0-\frac{1}{2}|+|x_1-\frac{1}{2}|$ in this case. Let $p'=(\frac{1}{2}, y_0)$ and $q' = (\frac{1}{2}, y_1)$, and define $\gamma_1$ be the path from p to $p'$, $\gamma_2$ will be from $p'$ to $q'$ and then $\gamma_3$ from $q'$ to $q$. We know by that for the curve that intersects the line and by Lemma $\ref{lem-MinimizingCurvesArePiecewiseLinear}$ that it is the case that minimizing the curve will be of the form $\ell_{p,p'}*\ell_{p'q'}*\ell_{q' q}$, then we can estimate

\begin{align} 
L_{\hat{f}_\infty}(\gamma) &= L_{\hat{f}_\infty}(\gamma_1) + L_{\hat{f}_\infty}(\gamma_2) + L_{\hat{f}_\infty}(\gamma_3) \label{eq-LB1}
\\& \ge  L_{\hat{f}_\infty}(\ell_{pp'}) + L_{\hat{f}_\infty}(\ell_{q'q})
\\&= |x_0 - \tfrac{1}{2}| + |x_1 - \tfrac{1}{2}|.\label{eq-LB2}
\end{align}

We can also estimate the upper bound

\begin{align} \label{eq-UB}
    d_{\hat{f}_\infty}(p,q) \le L_{\hat{f}_\infty}(\ell_{pp'})+L_{\hat{f}_\infty}(\ell_{q'q}) = |x_0 - \tfrac{1}{2}| + |x_1 - \tfrac{1}{2}|.
\end{align}

By taking the infimum in \eqref{eq-LB1}-\eqref{eq-LB2} we obtain

\begin{align} \label{eq-finalLB}
    d_{\hat{f}_\infty}(p,q) \ge \inf \{ |x_0 - \tfrac{1}{2}| + |x_1 - \tfrac{1}{2}|\}
\end{align}

Similarly by taking the infimum in \eqref{eq-UB}
we obtain

\begin{align} \label{eq-finalUB}
d_{\hat{f}_\infty}(p,q) \le \inf \{ |x_0 - \tfrac{1}{2}| + |x_1 - \tfrac{1}{2}| \}    
\end{align}

Hence by combining \eqref{eq-finalLB} and \eqref{eq-finalUB} we have the desired result.

\end{proof}

\begin{ex} \label{thm-Explicity Distance Function Half}
Let $p,q \in [0,1]^2$, $p = (x_1,y_1)$, $q=(x_2,y_2)$ and define $M = \{(\frac{1}{2},y) | \ y\in [0,1]\}$. We will define a class of curves  $\mathcal{M}=\{\phi_{pq} = \ell_{pp'} *\ell_{p'q'} * \ell_{q'q}: p',q' \in M\}$ so that the distance function defined with respect to the conformal factor
    \begin{align}
    \bar{f}_{\infty}(x,y)=
    \begin{cases}
       \frac{1}{2} & x=\frac{1}{2}, 0 \le y \le 1
        \\ 1 & \text{otherwise}
    \end{cases}
\end{align}
is given by
\begin{align}
    d_{\bar{f}_\infty}(p,q) &=\inf\left(\{  L(\ell_{pp'})+ \frac{1}{2}L(\ell_{p'q'})+L(\ell_{q'q}): \phi_{pq} \in \mathcal{M}\}\cup\{d(p,q)\}\right)
      \\&= \inf\left(\{  d(p,p')+ \frac{1}{2}d(p',q')+d(q',q) : p', q' \in M \}\cup\{d(p,q)\}\right).
\end{align}
\end{ex}

\begin{proof}
Let $p = (x_1,y_1)$, $q=(x_2,y_2)$. Then we will proceed by considering cases.

\textbf{Case 1 :} Let $p,q \in [0,\frac{1}{2}]\times [0,1]$. Suppose $\gamma$ to be a path connecting the points $p$ and $q$ such that $\gamma([0,1]) \cap (\{\frac{1}{2}\}\times[0,1])=\emptyset$. Note that if both $p,q \in M$ then this case does not apply. Now by Lemma $\ref{lem-MinimizingCurvesArePiecewiseLinear}$ the only path we need to consider is the straight line connecting $p$ to $q$ and hence
\begin{align}\label{eq-FirstLowerBound}
    L_{\bar{f}_\infty}(\gamma) \ge L_{\bar{f}_\infty}(\ell_{pq}) = d(p,q).
\end{align}

To demonstrate an upper bound we can calculate
\begin{align}\label{eq-FirstUpperBound}
 d_{\bar{f}_{\infty}}(p,q) \le L_{\bar{f}_\infty}(\ell_{pq}) =d(p,q).
\end{align}

Now consider a path $\gamma$, such that $\gamma([0,1]) \cap (\{\frac{1}{2}\}\times[0,1])\not =\emptyset$. Using Lemma \ref{lem-MinimizingCurvesArePiecewiseLinear}, we will consider the class of curves $\mathcal{M}$ to estimate a lower bound
\begin{align}\label{eq-secondLowerBound}
L_{\bar{f}_\infty}(\gamma) \ge L_{\bar{f}_\infty}(\phi_{pq})=d(p,p')+ \frac{1}{2}d(p',q')+d(q',q).
\end{align}

To show an upper bound, we can do the following. 
\begin{align}\label{eq-SecondUpperBound}
   d_{\bar{f}_{\infty}}(p,q)\le L_{\bar{f}_\infty}(\phi _{pq}) =   d(p,p')+ \frac{ 1}{2}d(p',q')+d(q',q).
\end{align}

By taking the inf over \eqref{eq-FirstLowerBound} with \eqref{eq-secondLowerBound} we find
\begin{align}
  d_{\bar{f}_\infty}(p,q) \ge\inf\left(\{  d(p,p')+ \frac{1}{2}d(p',q')+d(q',q): p',q' \in M\}\cup\{d(p,q)\}\right).
\end{align}

By combining \eqref{eq-FirstUpperBound} with \eqref{eq-SecondUpperBound} we find
\begin{align}
    d_{\bar{f}_{\infty}}(p,q) \le  \inf\left(\{  d(p,p')+ \frac{1}{2}d(p',q')+d(q',q): p',q' \in M\}\cup\{d(p,q)\}\right).
\end{align}

Both of these estimates together imply the desired result.

\textbf{Case 2:} Let $p,q \in [\frac{1}{2}, 1] \times [0,1]$ This case should be symmetric to case 1. 

\textbf{Case 3:} Let $p \in [0,\frac{1}{2}] \times [0,1]$ and $q \in [\frac{1}{2}, 1] \times [0,1]$. Let $\gamma$ be a path   connecting $p$ and $q$ and take $\mathcal{M}$ be a class of curves defined earlier. Notice that it will always be true that $\gamma([0,1]) \cap (\{\frac{1}{2}\}\times[0,1])\not =\emptyset$. Notice that if we choose $p'=q'$ then there is a choice of $p'$ so that the distance between $p$ and $q$ can be described as, $d(p,q) = d(p,p')+ \frac{1}{2}d(p',q')+d(q',q)$ which is the length of a curve in the class $\mathcal{M}$. By Lemma \ref{lem-MinimizingCurvesArePiecewiseLinear}, we can calculate a lower bound 

\begin{align} \label{eq-c3lowerbound}
    L_{\bar{f}_\infty}(\gamma) \ge L_{\bar{f}_\infty}(\phi_{pq})= d(p,p')+ \frac{1}{2}d(p',q')+d(q',q).
\end{align}

We can also obtain an upper bound, 

\begin{align} \label{eq-c3Upperbound}
    d_{\bar{f}_\infty}(p,q) \le L_{\bar{f}_\infty}(\phi_{pq}) = d(p,p')+ \frac{1}{2}d(p',q')+d(q',q).
\end{align}

By taking the inf over \eqref{eq-c3lowerbound} we obtain

\begin{align} \label {eq-finalc3lowerbound}
     d_{\bar{f}_\infty}(p,q) \ge \inf \{  d(p,p')+ \frac{1}{2}d(p',q')+d(q',q):p',q'\in M\}. 
\end{align}

Similarly by taking the inf over \eqref{eq-c3Upperbound} we obtain

\begin{align} \label{eq-finalc3Upperbound}
     d_{\bar{f}_\infty}(p,q) \le \inf \{  d(p,p')+ \frac{1}{2}d(p',q')+d(q',q):p',q'\in M\}. 
\end{align}

By combining \eqref{eq-finalc3lowerbound} and \eqref{eq-finalc3Upperbound} we have the desired results.

\end{proof}

\section{Examples of Sequences of Metric Spaces}\label{sec-Examples}

In this section we expand upon the examples considered by the first named author and C. Sormani in \cite{Allen-Sormani, Allen-Sormani-2}. In partciular, the examples are chosen to further probe the assumption of $C^0$ converence from below of the first named author, R. Perales, and C. Sormani \cite{Allen-Perales-Sormani} as well as to better understand the effects of assumptions on Lipschitz bounds.

\subsection{Blow Up Examples}\label{subsec-Blowup Examples}
In this subsection we consider metrics which blow up on regions. Due to the fact that the distance function takes the infimum over the lengths of all curves, one expects to allow for wilder blow ups of the conformal factor while still being able to control the distance function.

\subsubsection{Blow up on a Disappearing Center Square}
Here we see that when a sequence of conformal factors blows up on a disappearing square the blow up does not effect the limiting distance function.

\begin{thm}\label{ex-Blowup Uniform no Lipschitz} 
    For the sequence of functions defined by \begin{align}
    g_n(x,y)=
    \begin{cases}
       n^{\alpha} & \frac{n-1}{2n}\le x \le\frac{n+1}{2n}, \frac{n-1}{2n}\le y \le \frac{n+1}{2n} 
        \\ 1 & \text{otherwise}
    \end{cases}
\end{align}
where $\alpha \in (0,1)$ we find that $d_{g_n}$ converges uniformly to the Euclidean metric $d$   and does not satisfy a Lipschitz bound.
\end{thm}
\begin{proof}
We know $d(p,q)\leq d_{g_n}(p,q)$ by Theorem \ref{thm-Distance Lower Bound Estimate} since $1\leq g_n(x,y)$ for all $(x,y) \in [0,1]^2$.

\textbf{Case 1: $(p,q) \in \{(a,b)\in[0,1]^2\times[0,1]^2: \ell_{ab} \cap \{(\frac{1}{2},\frac{1}{2})\} = \emptyset\}$} 

 Notice that since $[\frac{n-1}{2n},\frac{n+1}{2n}] = [\frac{1}{2}-\frac{1}{2n},\frac{1}{2}+\frac{1}{2n}]$, then $(\frac{1}{2},\frac{1}{2})$ will be the only point in the set for every $n\in \N$. So there is some $N \in \N$ large enough so that for all $n \ge N$ we find $\ell_{pq}\cap[\frac{n-1}{2n},\frac{n+1}{2n}]^2=\emptyset$. Hence for $n \ge N$ we find that the length of the line connecting $p$ to $q$ is given by
 \begin{align}
     L_{g_n}(\ell_{pq})=L(\ell_{pq})=d(p,q).
 \end{align}
 Since $d_{g_n}(p,q) \le L_{g_n}(\ell(t))\leq d(p,q)$, we have obtained the desired upper bound in this case.
 
\textbf{Case 2: $(p,q) \in \{(a,b)\in[0,1]^2\times[0,1]^2: \ell_{ab} \cap \{(\frac{1}{2},\frac{1}{2})\} \not= \emptyset\}$} 

 If neither $p$ or $q$ are the point $(\frac{1}{2},\frac{1}{2})$ then, consider the length of the path $\phi(t)=(t,t), \quad t\in [\frac{n-1}{2n},\frac{n+1}{2n}]$ connecting the sequence of points $c_n=(\frac{n-1}{2n},\frac{n-1}{2n})$ with  $d_n=(\frac{n+1}{2n},\frac{n+1}{2n})$. We can calculate this distance to be $L_{g_n}(\phi) = \frac{2n^{\alpha}}{n} = C_n$. Denote $r_{0,n} = (a_{0,n},b_{0,n})$ to be the first time $\ell_{pq}$ enters $[\frac{n-1}{2n}, \frac{n+1}{2n}]^2$ and $r_{1,n}=(a_{1,n},b_{1,n})$ be the point of $\ell_{pq}$ exiting this region. Since for some large $N$ we know $p$ and $q$ aren't in $[\frac{n-1}{2n},\frac{n+1}{2n}]$, we know the path $\ell_{pr_{0,n}},\ell_{r_{1,n}q}\cap (\frac{n-1}{2n},\frac{n+1}{2n})^2 = \emptyset$. By definition of $r_{0,n}$, we know $a_{0,n}, b_{0,n}, a_{1,n}, b_{1,n} \in [\frac{n-1}{2n},\frac{n+1}{2n}]$ $\forall n \in\N$ so by the Squeeze Theorem, we know these sequences converge to a $\frac{1}{2}$. So then  $r_{0,n},r_{1,n}\rightarrow (\frac{1}{2},\frac{1}{2})$. Using this fact, we can see that $L(\ell_{pr_{0,n}}(t))+ C_n+L(\ell_{r_{1,n}q}(t)) \rightarrow d$ by
\begin{align}
    L_{g_n}(\ell_{pq}) &= L(\ell_{pr_{0,n}}) + L_{g_n}(\ell_{r_{0,n}r_{1,n}}) + L(\ell_{r_{1,n}q})
    \\&\leq L(\ell_{pr_{0,n}})+ C_n+L(\ell_{r_{1,n}q})
    \\&=d(p,q) - \left(\sqrt{|x_0-\frac{1}{2}|^2+|y_0-\frac{1}{2}|^2}+\sqrt{|x_1-\frac{1}{2}|^2+|y_1-\frac{1}{2}|^2}\right) + L(\ell_{pr_{0,n}})
    \\&\quad + C_n+L(\ell_{r_{1,n}q})
    \\&=d(p,q) + \left|\sqrt{|x_0-a_{0,n}|^2+|y_0-b_{0,n}|^2} - \sqrt{|x_0-\frac{1}{2}|^2+|y_0-\frac{1}{2}|^2}\right| + C_n + 
    \\&\quad \left|\sqrt{|x_1-a_{1,n}|^2+|y_1-b_{1,n}|^2} - \sqrt{|x_1-\frac{1}{2}|^2+|y_1-\frac{1}{2}|^2}\right|
    \\&= d(p,q) + D_n
\end{align}
where $D_n \rightarrow 0$. 

Without loss of generality, let $q=(\frac{1}{2},\frac{1}{2})$, then only $r_{0,n}$ exists and still converges to $(\frac{1}{2},\frac{1}{2})$. We then can estimate an upperbound.
\begin{align}
    L_{g_n}(\ell_{pq}) &= L(\ell_{pr_{0,n}})+L_{f_n}(\ell_{r_{0,n}q})
    \\ &\leq L(\ell_{pr_{0,n}}) + C_n
    \\ &=d(p,q) + \left|\sqrt{|x_0-a_{0,n}|^2+|y_0-b_{0,n}|^2} - \sqrt{|x_0-\frac{1}{2}|^2+|y_0-\frac{1}{2}|^2}\right| + C_n
    \\&=d(p,q) + D_n
\end{align}
Now by Theorem \ref{thm-Squeeze} we find that $d(p,q)\leq d_{g_n}(p,q) \leq d(p,q) + D_n$ showing that $d_{f_n}\rightarrow d$ uniformly. 

Now we will show that the sequence does not satisfy an upper Lipschitz bound. Consider $p=\left( \frac{1}{2},\frac{1}{2}\right)$, $q_n=\left( \frac{n+1}{2n},\frac{1}{2}\right)$. To show that $d_{g_n}$ does not satisfy the Lipschitz upper bound, we show that 
   \begin{align}
       \forall c >0,\  \ \exists p,q \in [0,1]^2, \text{such that } \  cd(p,q) < d_{g_n}(p,q) 
   \end{align}

   To this end, consider $c_n= \frac{n^\alpha}{2}$ and we have   $d(p,q_n) = \left |\frac{n+1}{2n} - \frac{1}{2}\right | = \frac{1}{2n}$ and hence we find
    \begin{align} 
      c_n \ d(p,q_n)=  \frac{n^{\alpha}}{2} \left(\frac{1}{2n} \right) = \frac{1}{4n^{1-\alpha}} < \frac{1}{2 n^{1-\alpha}}=d_{g_n}(p,q_n),
    \end{align}

    which implies that $d_{g_n}$ does not satisfies the upper Lipschitz upper bound.     
\end{proof}

Now we will see that if we increase the blow up rate anymore than Example \ref{ex-Blowup Uniform no Lipschitz} then the sequence will not converge uniformly to Euclidean space. This was already observed by the first named author and C. Sormani in Example 3.5 and 3.6 of \cite{Allen-Sormani-2} where further details can by found.
\begin{thm}\label{ex-Blowup No Uniform}
    For the sequence of functions defined by \begin{align}
    \hat{g}_n(x,y)=
    \begin{cases}
       n^{\alpha}& \frac{n-1}{2n}\le x \le\frac{n+1}{2n}, \frac{n-1}{2n}< y \le \frac{n+1}{2n} 
        \\ 1 & \text{otherwise}
    \end{cases}
\end{align}
where $\alpha \in [1,\infty)$ we find that $d_{\hat{g}_n}$ does not converge uniformly to the Euclidean metric $d$.
\end{thm}

\begin{proof}
    Let the points $p = (\frac{1}{2},\frac{1}{2})$ and $q_n = (\frac{n+1}{2n},\frac{1}{2})$, so the euclidean distance is defined as $d(p,q_n) =\left| \frac{n+1}{2n} - \frac{1}{2} \right| = \frac{1}{2n}$  and we also have defined $d_{\hat{g}_n}(p, q_n) = n^{\alpha}  d(p,q_n) = \frac{n^{\alpha-1}}{2}$
    
Now if we calculate
    \begin{align}
        \sup_{p,q \in [0,1]^2}\left |d_{\hat{g}_n}(p,q) - d(p,q)\right| \ge \left| \frac{n^{\alpha-1}}{2} - \frac{1}{2n} \right|=\frac{n^{\alpha}-1}{2n},
    \end{align}
 and hence we find that $d_{\hat{g}_n} \not \rightarrow d$ uniformly.
\end{proof}

\subsubsection{Blow up on a Center Line}

Now we will see that if the sequence only blows up on a set of zero Lebesgue measure then this will not effect the distance function. This example is part of the test of what is required to show an upper Lipschitz bound since if the metric is completely unaffected by blowing up on a set of Lebesgue measure zero then an upper Lipschitz bound clearly applies.

\begin{thm}\label{thm-bar{g}_n uniform convergence} 
    For the sequence of functions defined by \begin{align}
    \bar{g}_{n}(x,y)=
    \begin{cases}
       n^{\alpha} & x=\frac{1}{2}, 0\le y \le 1 
        \\ 1 & \text{otherwise}
    \end{cases}
\end{align}
where  $\alpha \in (0,\infty)$ we find that $d_{\bar{g}_n}=d$ for all $n \in \N$.
\end{thm}

\begin{proof}

Notice that by Theorem \ref{thm-Distance Lower Bound Estimate} since $\bar{g}_n\geq 1$ we find that $d_{\bar{g}_n}(p,q)\geq d(p,q)$.

\textbf{Case 1: $(p,q)\in \{(a,b)\in[0,1]^2\times [0,1]^2:\ell_{ab}\cap(\frac{1}{2}\times[0,1])=\emptyset\}$}
Then by calculating the length of any path $\gamma$ in this region, we know $d_{\bar{g}_n}(p,q)\leq L_{\bar{g}_n}(\gamma)=L(\gamma)=d(p,q)$.

\textbf{Case 2: $(p,q)\in \{(a,b)\in [0,1]^2\times [0,1]^2:\ell_{ab}\cap(\frac{1}{2}\times[0,1])\not=\emptyset\}$} 

Suppose $p= (x_0,y_0)$ $q = (x_1,y_1)$ where $x_0 \in[0,\frac{1}{2}]$ and $x_1\in (\frac{1}{2},1]$.  Then we know gamma intersects $\left\{\tfrac{1}{2}\right\}\times[0,1]$ at one point. Since this point is of measure 0, we know our length isn't affected by $n^\alpha$. So the length of $d_{\bar{g}_n}(p,q)\leq L_{\bar{g}_n}(\ell_{pq})=L(\ell_{pq})=d(p,q)$.

The argument is symmetric if $x_1\in [0,\frac{1}{2}]$ and $x_0 \in (\frac{1}{2},1]$.

Suppose $p= (\frac{1}{2},y_0)$ $q = (\frac{1}{2},y_1)$, let $a_0=(\epsilon+\frac{1}{2},y_0)$ and $a_1=(\epsilon + \frac{1}{2},y_1)$ for some $ \epsilon>0$.
Define
\begin{align}
    \phi=\ell_{pa_0}*\ell_{a_0a_1}*\ell_{a_1q}
\end{align}
we can calculate its length to be
\begin{align}
   d_{\bar{g}_n}(p,q)&\le L_{\bar{g}_n}(\phi)
   \\&=L(\phi)
    \\&=2\epsilon+|y_0-y_1|=2\epsilon+d(p,q).
\end{align}
Hence we have shown that $d(p,q)\leq d_{\bar{g}_n}(p,q)\leq d(p,q)+2\varepsilon$ and by letting $\varepsilon\rightarrow 0$ we find the desired result.
\end{proof}

\subsection{Shortcut Examples}\label{subsec-Shortcut Examples}

In this subsection we consider metrics which form shortcuts on different regions. Due to the fact that the distance function takes the infimum over the lengths of all curves, one expects even mild shortcuts to have a drastic effect on the limiting distance function.

\subsubsection{Shortcut on a Disappearing Center Square}
Here we see that if the shortcut is forming along a disappearing square then it will not be enough to effect lower bounds on the distance function but it will be enough to effect a lower Lipschitz bound.

\begin{thm}\label{ex-Shortcut on Dissappearing Square} 
    Let $\alpha \in (0,\infty)$ and consider the sequence of functions defined by \begin{align}
    f_n(x,y)=
    \begin{cases}
       \frac{1}{n^{\alpha}} & \frac{n-1}{2n}\le x \le\frac{n+1}{2n}, \frac{n-1}{2n} \le y \le \frac{n+1}{2n} 
        \\ 1 & \text{otherwise}
    \end{cases}
\end{align}
We find that $d_{f_n}$ converges uniformly to the Euclidean metric $d$ but does not satisfy a lower Lipschitz bound.
\end{thm}
\begin{proof}
First notice that $d(p,q)\geq d_{f_n}(p,q)$ by Theorem \ref{thm-Distance Lower Bound Estimate} since $1\geq f_n(x,y)$ for all $(x,y) \in [0,1]^2$.

    Now fix a p,q such that $p,q \not = (\frac{1}{2},\frac{1}{2})$, consider a piecewise linear path $\gamma_{pq}$ which is linear wherever $f_n$ is constant and so that $\gamma_{pq} \cap (\frac{n-1}{2n}, \frac{n+1}{2n})^2\not = \emptyset$. Then we can calculate its length to be $L_{f_n}(\gamma_{pq})=d(p,p'_1) + d_{f_n}(p'_1,p'_2) + d(p'_2,q)$ where $p'_1,p'_2 \in \partial (\frac{n-1}{2n},\frac{n+1}{2n})^2 $.

    
    Consider a second path $\ell_{pq}$ where  $L(\ell_{pq})$ is defined as $d(p,q)$ then we can do the following 
    \begin{align}
 L_{f_n}(\ell_{pq})&\le d(p,q)
 \\&\le d(p,p'_1) + d(p'_1,p'_2) + d(p'_2,q)
  \\&\le L_{f_n}(\gamma_{pq}) + |d(p'_1,p'_2)-d_{f_n}(p'_1,p'_2)|
  \\&= L_{f_n}(\gamma_{pq}) + \frac{\sqrt{2}(n^{\alpha-1}-1)}{n^\alpha}
   \\&\le L_{f_n}(\gamma_{pq}) + \frac{\sqrt{2}}{n^\alpha}. 
\end{align}
Furthermore, we can obtain
\begin{align}
 L(\ell_{pq}) \le L_{f_n}(\ell_{pq}) +\frac{\sqrt{2}}{n^\alpha}.
\end{align}
By combining both of these observations we find
\begin{align}
    d(p,q)&= L(\ell_{pq})\le L_{f_n}(\gamma_{pq}) +\frac{\sqrt{2}}{n^\alpha}
\end{align}
By taking the inf on the right side and using Lemma \ref{lem-MinimizingCurvesArePiecewiseLinear} we get 
\begin{align}
    d(p,q) \le d_{f_n}(p,q) + \frac{\sqrt{2}}{n^\alpha}.
\end{align}

Then we can rewrite this as 
\begin{align}
d(p,q) -\frac{\sqrt{2}}{n^\alpha} \le d_{f_n}(p,q) \le d(p,q) .
\end{align}
then by the Theorem \ref{thm-Squeeze} $d _{f_n}\rightarrow d$ uniformly as $\frac{\sqrt{2}}{n^\alpha} \rightarrow 0$

   To show that $d_{f_n}$ does not satisfy the Lipschitz lower bound, we show that 
   \begin{align}
       \forall c >0,\  \ \exists p,q \in [0,1]^2, \text{such that } \  cd(p,q) > d_{f_n}(p,q) 
   \end{align}

To this end, pick $c_n = \frac{2}{n^\alpha}$ and consider $p =(\frac{1}{2},\frac{1}{2})$ and $q_n=(\frac{n-1}{2n}, \frac{1}{2})$,

\begin{align}
    d(p,q_n) = \left | \frac{n}{2n} - \frac{n-1}{2n} \right|  = \frac {1}{2n}
   \\
   d_{f_n}(p,q_n) = \frac{1}{n^\alpha} \ d(p,q) = \frac{1}{n^\alpha} \left (\frac{1}{2n} \right ) = \frac{1}{2n^{\alpha+1}}
\end{align}

Combining previous results we have the following

\begin{align}
   c_n d(p,q) = \frac{2}{n^\alpha} \left (\frac{1}{2n} \right)= \frac{1}{n^{\alpha+1}} > \frac{1}{2n^{\alpha+1}}= d_{f_n}(p, q_n).
\end{align}

Hence $d_{f_n}$ does not satisfy a Lipschitz lower bound. 
\end{proof}

\subsubsection{Shortcut Along A Curve of Hausdorff $1$ Measure Zero}

Here we see that if the Hausdorff $1$ measure of the curve along which the shortcut is forming is zero then the shortcut will not effect the distance function. This is not too surprising but is included for comparison purposes.

\begin{thm}\label{ex-Shortcut H1 Measure Zero} 
    For the sequence of functions defined by \begin{align}
    f_n(x,y)=
    \begin{cases}
       \frac{1}{n} & x=\frac{1}{2}, 0 \le y \le 1, y \in \Q
        \\ 1 & \text{otherwise}
    \end{cases}
\end{align}
we find that $d_{f_n}=d$.
\end{thm}
\begin{proof}
   Consider any path $\gamma$ connecting $p$ to $q$. Denote $R = \gamma([0,1]) \cap (\{\frac{1}{2}\}\times\Q)$. Notice that $|R|\leq |\Q |$ showing $R$ is countable. So the length of $\gamma$ is unaffected by $f_n$ defined on $\{\frac{1}{2}\}\times \Q$. Then we can calculate the length of the path.
   \begin{align}
       L_{f_n}(\gamma)&=L_{f_n}(R\cap \gamma)+ L_{f_n}(R^C\cap \gamma)
       \\&=0+L(R^C\cap \gamma)=L(\gamma).
   \end{align}
   Taking the infimum on both sides, we arrive at the desired equality
   \begin{align}
       d_{f_n}(p,q)&=d(p,q).
   \end{align}
\end{proof}

\subsubsection{Shortcut Along A Rectangle Containing The Center Line}

Here we give another instance of what was observed by the first named author and C. Sormani \cite{Allen-Sormani,Allen-Sormani-2}, that a shortcut along a curve has a drastic effect on the limiting distance function.

\begin{thm}\label{ex-Shortcut on a Rectangle} 
    For the sequence of functions defined by \begin{align}
    \hat{f}_n(x,y)=
    \begin{cases}
       \frac{1}{n} & \frac{n-1}{2n}\le x \le\frac{n+1}{2n}, 0 \le y \le 1
        \\ 1 & \text{otherwise}
    \end{cases}
\end{align}
we find that $d_{\hat{f}_n}$ converges uniformly to $d_{\hat{f}_\infty}$.
\end{thm}

\begin{proof}
Let $p=(x_0,y_0)$, $q=(x_1,y_1)$

\textbf{Case 1: $p\in [0,\frac{1}{2}]\times [0,1]$ and $q\in [\frac{1}{2},1]\times[0,1]$.} 
We know $d_{\hat{f}_\infty}(p,q)=|x_0-\frac{1}{2}|+|x_1-\frac{1}{2}|$. Let $a_0=(\frac{1}{2},y_0)$, and $a_1=(\frac{1}{2},y_1)$. Consider the path 
\begin{align}
    \beta = \ell_{p,a_0}*\ell_{a_0,a_1}*\ell_{a_1,q}.
\end{align}
We can calculate its length as follows
\begin{align}
    L_{f_n}(\beta)&=L_{f_n}(\ell_{p,a_0}) + L_{f_n}(\ell_{a_0,a_1}) + L_{f_n}(\ell_{a_1,q})
    \\&=(|x_0-\frac{n-1}{2n}|+\frac{1}{n}|\frac{n-1}{2n}-\frac{1}{2}|) + \frac{|y_0-y_1|}{n}+(|x_1-\frac{n+1}{2n}|+\frac{1}{n}|\frac{n+1}{2n}-\frac{1}{2}|)
    \\&\leq |x_0-\frac{1}{2}| + |x_1-\frac{1}{2}| + \frac{|y_0-y_1|}{n}
    \\&\le d_{\hat{f}_\infty}(p,q)+\frac{1}{n}.
\end{align}
So we have an upper bound. For the lower bound, any curve connecting $p$ and $q$ must be less than or equal to $|x_0-\frac{n-1}{2n}|$ and $|x_1-\frac{n+1}{2n}|$ since this is the shortest length within these regions.
\begin{align}
    d_{\hat{f}_n}(p,q) &\ge |\frac{n-1}{2n}-x_0| + |\frac{n+1}{2n}-x_1|
    \\&=|\frac{n-1}{2n}-x_0| + |\frac{n+1}{2n}-x_1| + |\frac{1}{2}-x_0| + |\frac{1}{2}-x_1|- |\frac{1}{2}-x_0| - |\frac{1}{2}-x_1|
    \\&=d_{\hat{f}_\infty}(p,q)-(|\frac{1}{2}-x_1|-|x_1-\frac{n+1}{2n}|)- |\frac{1}{2}-x_0|-|x_0-\frac{n-1}{2n}|
    \\&=d_{\hat{f}_\infty}(p,q)-C_n'.
\end{align}
\textbf{Case 2: $p,q\in [0,\frac{1}{2}]\times [0,1]$}

Let $b_1 = (\frac{n-1}{2n},y_1)$ $b_0= (\frac{n-1}{2n},y_0)$. Consider the paths $\ell_{p,q}$ and
\begin{align}
    \phi = \ell_{p,b_0}*\ell_{b_0,b_1}*\ell_{b_1,q}.
\end{align}
If neither $p$ or $q$ are $\{\frac{1}{2}\}\times[0,1]$, then we know there exists some $N$ such that $p,q\notin [\frac{n-1}{2n},\frac{1}{2}]\times[0,1]$ and hence we can calculate $L_{\hat{f}_n}(\ell_{pq})=d(p,q)$ and $L_{\hat{f}_n}(\phi) = \frac{|y_0-y_1|}{n}+|x_0-\frac{1}{2}|+|x_1-\frac{1}{2}|\le|x_0-\frac{1}{2}|+|x_1-\frac{1}{2}|+\frac{1}{n}$. This implies that 
\begin{align}
 d_{\hat{f}_n}(p,q) &\le \min\{L_{\hat{f}_n}(\ell_{pq}) ,L_{\hat{f}_n}(\phi)\}  
 \\=& \min\{d(p,q) ,|x_0-\frac{1}{2}|+|x_1-\frac{1}{2}|+\frac{1}{n}\}  
  \\\le& \min\{d(p,q)+\frac{1}{n} ,|x_0-\frac{1}{2}|+|x_1-\frac{1}{2}|+\frac{1}{n}\}  
 \\\le& \min\{d(p,q) ,|x_0-\frac{1}{2}|+|x_1-\frac{1}{2}|\} +\frac{1}{n}=d_{\hat{f}_{\infty}}(p,q)+\frac{1}{n}.
\end{align}

For the lower bound, if we suppose $\gamma\cap(\{\frac{1}{2}\}\times [0,1])=\emptyset$ then for $N$ chosen large enough the shortest path in this class is $\ell_{p,q}$. So now we consider when $\gamma\cap(\{\frac{1}{2}\}\times [0,1])\not=\emptyset$ and define $r_{0,n}=(\frac{n-1}{2n},h_0)$ to be the first time $\gamma$ intersects the region where $\hat{f}_n=\frac{1}{n}$, and $r_{1,n}=(\frac{n-1}{2n},h_1)$ be the last time $\gamma$ intersects this region. Also denote $t_0$, $t_1 \in [0,1]$ to be the times where $\gamma(t_0)= r_{0,n}$ and $\gamma(t_1)= r_{1,n}$. Then we can calculate the length of $\gamma$ to be 
\begin{align}
    L_{\hat{f}_n}(\gamma)&=L(\gamma|_{[0,t_0]})+L_{\hat{f}_n}(\gamma|_{[t_0,t_1]})+ L(\gamma|_{[t_1,1]})
    \\&\geq L(\ell_{p,r_{0,n}})+L_{\hat{f}_n}(\ell_{r_{0,n},r_{1,n}})+L(\ell_{r_{1,n},q})
    \\&\geq |x_0-\frac{n-1}{2n}|+\frac{1}{n}|h_0-h_1|+|\frac{n-1}{2n}-x_1|
    \\&=(|x_0-\frac{n-1}{2n}|-|x_0-\frac{1}{2}|)+\frac{1}{n}|h_0-h_1|+(|\frac{n-1}{2n}-x_1|-|\frac{1}{2}-x_1|)
    \\&\quad + |x_1-\frac{1}{2}|+|\frac{1}{2}-x_0|
    \\&\ge |x_1-\frac{1}{2}|+|\frac{1}{2}-x_0| +(|\frac{n-1}{2n}-x_1|-|\frac{1}{2}-x_1|)+(|x_0-\frac{n-1}{2n}|-|x_0-\frac{1}{2}|)
    \\&=|x_1-\frac{1}{2}|+|\frac{1}{2}-x_0|-C_n'.
\end{align}
Since any curve connecting $p$ to $q$ falls into one of these two cases we take the minimum between these two lower bounds to find 
\begin{align}
    d_{\hat{f}_n}(p,q)&\geq\min\{d(p,q),|x_0-\frac{1}{2}|+|x_1-\frac{1}{2}|-C_n'\}
    \\&\ge\min\{d(p,q)-C_n',|x_0-\frac{1}{2}|+|x_1-\frac{1}{2}|-C_n'\}
    \\&= \min\{d(p,q),|x_0-\frac{1}{2}|+|x_1-\frac{1}{2}|\}-C_n'
    \\&=d_{\hat{f}_{\infty}}(p,q)-C_n'
\end{align}
 and we have a lower bound. 
 
 Putting both cases together, let $C_n=\max\{C_n',\frac{1}{n}\}$ and hence we can bound the function
 \begin{align}
     d_{\hat{f}_\infty}(p,q)-C_n\leq d_{\hat{f}_n}(p,q) \leq d_{\hat{f}_\infty}(p,q)+C_n,
 \end{align}
 so that the desired result follows from Theorem \ref{thm-Squeeze}.
\end{proof}

In the next example we consider a similar shortcut forming along a curve but we delete portions of the shortcut along the curve to consider at what rate the shortcut must appear in order to effect the limiting distance function.

\begin{thm}\label{ex-Shortcut Dense Uniform no Lipschitz} 
For the sequence of functions defined by
 \begin{align}
 S=\left\{s_{i,j}= \tfrac{i}{2^j}\,: \,  i=1,2,... (2^j-1),\, j\in \mathbb{N}\right\}= \left\{ \tfrac{1}{2},\tfrac{1}{4}, \tfrac{2}{4},  \tfrac{3}{4}, 
\tfrac{1}{8},\tfrac{2}{8},\tfrac{3}{8} ,...\right\}
 \end{align}
 which is dense in $[0,1]$
 and
 \begin{align}
 \delta_j:= \left (\frac{1}{2} \right)^{2j},\quad j \in \mathbb{N}.
 \end{align}
 Let
  \begin{align}
 \tilde{f}_j(x,y)=
 \begin{cases}
  \frac{1}{j} & (x,y)\in [\frac{1}{2}-\delta_j,\frac{1}{2}+\delta_j]\times[s_{i,j}-\delta_j, s_{i,j} +\delta_j] \textrm{ for } i =1,...,2^j-1
 \\ 1 & \textrm{ elsewhere }
 \end{cases}
\end{align}
    we find that $d_{\tilde{f}_n}$ converges uniformly to the Euclidean metric $d$ but does not satisfy a lower Lipschitz bound.
  
\end{thm}
\begin{proof}
Notice that $\tilde{f}_j \le 1$ and hence by Theorem \ref{thm-Distance Lower Bound Estimate} we know that $d_{\tilde{f}_j} \le d$. 

We will now establish the lower bound by cases. Let $p,q \in [0,1]^2$ so that $p = (x_0,y_0)$ and $q=(x_1,y_1)$.

\textbf{Case 1: $p$, $q\in [0,\frac{1}{2})\times[0,1]$, or $p\in [0,\frac{1}{2})\times[0,1]$ $q\in \{\frac{1}{2}\}\times[0,1]$}    
    
     Take some path $\gamma$ such that $(\gamma([0,1])\setminus (\{p\}\cup \{q\}))\cap (\{\frac{1}{2}\}\times[0,1])=\emptyset$. Then amongst curves of this type, we know the minimizing curve is $\ell_{pq}$ whose length is $d(p,q)$. 
    
    Consider some $\gamma$ such that $\gamma([0,1])\cap(\{\frac{1}{2}\}\times[0,1])\not= \emptyset$. Denote $r_0=\gamma(t_0)=(x_{r_0},y_{r_0})$ to be the first time it intersects with $(\{\frac{1}{2}\}\times[0,1])$, and $r_1=\gamma(t_1) = (x_{r_1},y_{r_1})$ be the last time it intersects this set. Also let $a_{0,j}=\gamma(s_{0,j})$ be the first time gamma intersects $\{\frac{1}{2}-\delta_j\}\times[0,1]$ for $t\in[0,t_0]$ and $a_{1,j}=\gamma(s_{1,j})$ be the last time $\gamma$ intersects $\{\frac{1}{2}-\delta_j\}\times[0,1]$ for $t\in [t_1,1]$. Calculating the length of $\gamma$ we find 
    \begin{align}
        L_{\tilde{f}_j}(\gamma)&= L(\gamma|_{[0,s_{0,j}]})+ L_{\tilde{f}_j}(\gamma|_{[s_{0,j},t_0]})+L_{\tilde{f}_j}(\gamma|_{[t_0,t_1]})+L_{\tilde{f}_j}(\gamma|_{[t_1,s_{1,j}]}) + L(\gamma|_{[s_{1,j},1]})
        \\&\geq L(\ell_{pa_{0,j}})+L_{\tilde{f}_j}(\gamma|_{[t_0,t_1]})+L(\ell_{a_{1,j}q})
        \\&= L(\ell_{pr_0}))+(L(\ell_{pa_{0,j}})- L(\ell_{pr_0}))+L_{\tilde{f}_j}(\gamma|_{[t_0,t_1]})
        \\&\quad+(L(\ell_{a_{1,j}q})-L(\ell_{r_1q}))+L(\ell_{r_1q})
        \\&= L(\ell_{pr_0})-B_j+L_{\tilde{f}_j}(\gamma|_{[t_0,t_1]})-B_j'+L(\ell_{r_1q})
        \\& \geq L(\ell_{pr_0})-B_j+L_{\tilde{f}_j}(\gamma|_{[t_0,t_1]})-B_j'+L(\ell_{r_1q}).
    \end{align}
With $L(\ell_{pa_{0,j}})-L(\ell_{pr_0}) = B_j\rightarrow 0$, and $ L(\ell_{a_{1,j}1})-L(\ell_{r_1}q)=B_j'\rightarrow 0$. 

Now we want to get a better estimate for $L_{\tilde{f}_j}(\gamma|_{[t_0,t_1]})$ . To this end, we claim that we only have to consider paths $\psi_j$ from $r_0$ to $r_1$ such that $\psi_j\subset [\frac{1}{2}-\delta_j,\frac{1}{2}+\delta_j]\times[0,1]$ $\forall j \in \N$. 

To prove this claim, let $\psi_j$ be a path connecting $r_0$ and $r_1$. Suppose $\psi_j([0,1]) \cap ([\frac{1}{2}-\delta_j,\frac{1}{2}+\delta_j]\times[0,1])^c\not= \emptyset$. Define $E_j=\{e_{0,j},...,e_{m,j}\}$ to be the set of points in which $\psi_j$ exits and enters, consecutively, the set $[\frac{1}{2}-\delta_j,\frac{1}{2}+\delta_j]\times[0,1]$ that lie on its boundary. If we define $c_{i,j}$ to be $\psi (c_{i,j}) = e_{i,j}$ then we can define a path 
\begin{align}
\phi_j(t)=\
    \begin{cases}
        \psi_j(t),& \text{if } t\in [0, c_{i,j}) \cup (c_{i+1,j}, c_{i+2,j}), \text{for } i\in \{0,1,...,m-2\}
        \\\ell_{e_{i,j}e_{i+1,j}},& \text{if } t\in [c_{i,j},c_{i+1,j}], \text{ for } i \in \{0,1,...,m-1\}
    \end{cases}
\end{align}

which is contained in the $[\frac{1}{2}-\delta_j,\frac{1}{2}+\delta_j]\times[0,1]$. Then we can calculate 
\begin{align}
    L_{\tilde{f}_j}(\psi_j)&=L_{\tilde{f}_j}(\psi_j\cap([\frac{1}{2}-\delta_j,\frac{1}{2}+\delta_j]\times[0,1]))+L(\psi_j\cap([\frac{1}{2}-\delta_j,\frac{1}{2}+\delta_j]\times[0,1])^c)
    \\&\geq L_{\tilde{f}_j}(\psi_j\cap([\frac{1}{2}-\delta_j,\frac{1}{2}+\delta_j]\times [0,1]) +\sum_{k=1}^m d(e_{k-1,j},e_{k,j}) 
    \\&\geq L_{\tilde{f}_j}(\psi_j\cap([\frac{1}{2}-\delta_j,\frac{1}{2}+\delta_j]\times [0,1]) +\sum_{k=1}^mL_{\tilde{f}_j}(\ell_{e_{k-1,j}e_{k,j}})
    \\&=L_{\tilde{f}_j}(\phi_j).
\end{align}

Hence, we have established the claim and we only need to consider curves from $r_0$ to $r_1$ such that $\psi_j\subset [\frac{1}{2}-\delta_j,\frac{1}{2}+\delta_j]\times[0,1]$ $\forall j \in \N$. Furthermore, we may assume that $\psi_j$ is monotone non-decreasing in $y$ since we could always find a monotone curve  which is shorter in length than a curve which is non-monotone in $y$. 

 With this in mind, define another set $U_j=\{u_{k_1,j},u_{k_2,j},...,u_{k_b,j}\}$ to be the consecutive intersections of $\psi_j$ with some $([\frac{1}{2}-\delta_j,\frac{1}{2}+\delta_j]\times \{s_{k,j}-\delta_j,s_{k,j}+\delta_j\}) $. Also, define $v_{k_i,j}$ so that $\gamma(v_{k_i,j})=a_{k_i,j}$ for $i \in \{1,...,b\}$ and $\gamma(t_0)=u_{k_0,j}=\gamma(v_{k_{0,j}})$, $\gamma(t_1)=u_{k_{b+1},j}=\gamma(v_{k_{b+1,j}})$. We then define a path 
\begin{align}
    \tau_j=\ell_{r_0u_{k_1,j}}*...*\ell_{r_0u_{k_b,j}}
\end{align}
and by applying Lemma $\ref{lem-MinimizingCurvesArePiecewiseLinear}$ we know that $\tau_j$ is exactly the type of curve which needs to be considered when estimating the distance. So we can make the observation that $L_{\tilde{f}_j}(\psi_j)\geq L_{\tilde{f}_j}(\tau_j)$. Let $N_j=\frac{d(r_0,r_1)}{\frac{1}{2^j}}=2^jd(r_0,r_1)$ be an estimate of the amount of rectangles in between $r_0$ and $r_1$. Define $A_{i,j} = [\frac{1}{2}-\delta, \frac{1}{2} +\delta]\times[s_{i,j} -\delta_j]$ with $\displaystyle A_j = \bigcup_{0\leq i\leq 2^j-1} A_{i,j}$ Then we can use this to estimate the length of $\tau_j$. 

\begin{align}
    L_{\tilde{f}_j}(\tau_j) &= L_{\tilde{f}_j}(\tau_j \cap A_j) +  L_{\tilde{f}_j}(\tau_j\cap (A_j)^c)
    \\ &\geq L(\tau_j\cap (A_j)^c)
    \\&\geq \frac{1}{2^j}N_j=d(r_0,r_1)
\end{align}

We can then use this to bound $L_{\tilde{f}_j}(\gamma|_{[t_0,t_1]})$ from below
    \begin{align}
        L_{\tilde{f}_j}(\gamma|_{[t_0,t_1]})&\geq L_{\tilde{f}_j}(\tau_j)\ge d(r_0,r_1).\label{eq-Question?!}
    \end{align}
    Summing up the estimation of each segment, we arrive at a lower bound for the entire path. 
    \begin{align}
        L_{\tilde{f}_j}(\gamma)&>L(\ell_{pr_0}) + L(\ell_{r_0r_1})+L(\ell_{r_1q})-B_j-B'_j-C_j
        \\&=L(\ell_{pr_0}) + L(\ell_{r_0r_1})+L(\ell_{r_1q})-T_j
        \\&=d(p,r_0)+d(r_0,r_1)+d(r_1,q)-T_j,
    \end{align}
    where we have defined $T_j=B_j+B'_j+C_j$.
    Comparing this lower bound for paths that intersect $\{\frac{1}{2}\}\times[0,1]$ with $\ell_{pq}$, we notice  
    \begin{align}
    d(p,r_0)+d(r_0,r_1)+d(r_1,q)-T_j \geq d(p,q)-T_j.
    \end{align}
    giving us the lower bound of $d(p,q)-T_j\leq d_{\tilde{f}_j}(p,q)$

    \textbf{Case 2: $p$, $q\in (\frac{1}{2},1]\times[0,1]$} Follows by the same logic as in Case 1.
 
    \textbf{Case 3: $p\in [0,\frac{1}{2}]\times[0,1]$, $q\in [\frac{1}{2},1]\times[0,1]$}

    Consider $\gamma$ to be any path which connects $p$ and $q$. Let $\gamma(t_0)=a_{0,j} = (x_{0,j},y_{0,j})$ be the point in which $\gamma$ intersects $(\{\frac{1}{2}-\delta_j\}\times[0,1])$, and $\gamma(t_1)=a_{1,j}=(x_{1,j},y_{1,j})$ be the point in which $\gamma$ intersects $(\{\frac{1}{2}+\delta_j\}\times[0,1])$. 
    
    Define $r_{0,j} = (\frac{1}{2},y_{0,j})$ and $r_{1,j} = (\frac{1}{2},y_{1,j})$. Now some path $\gamma$ connecting points $p$ and $q$.
    \begin{align}
        L_{\tilde{f}_j}(\gamma)&= L_{\tilde{f}_j}(\gamma|_{[0,t_0]})+L_{\tilde{f}_j}(\gamma|_{[t_0,t_1]})+L_{\tilde{f}_j}(\gamma|_{[t_1,1]})
        \\ &\geq L(\ell_{p,{a_{0,j}}})+L_{\tilde{f}_j}(\gamma|_{[t_0,t_1]})+L(\ell_{a_{1,j}}q) 
        \\& = L(\ell_{p,r_{0,j}})+(L(\ell_{p,{a_{0,j}}})-L(\ell_{p,r_{0,j}}))+L_{\tilde{f}_j}(\gamma|_{[t_0,t_1]})
        \\&\quad +(L(\ell_{a_{1,j},q})-L(\ell_{r_{1,j},q}))+ L(\ell_{r_{1,j},q}) 
    \end{align}
    Using what was observed in case 1. We can approximate the length $L_{f_j}(\gamma|_{[t_0,t_1]})\geq d(r_0,r_1)$.
    \begin{align}
    L_{f_j}(\gamma)&\geq L(\ell_{pr_{0,j}})+d(r_0,r_1)+L(\ell_{r_{1,n}q})
    \\&=d(p,r_{0,j}) + d(r_{0,j},r_{1,j}) +d(r_{1,j},q)
    \\&\geq d(p,q),
    \end{align}
    where we use the triangle inequality in the last line. Since this is true for all possible paths, we get a lower bound for the distance of these paths. 

    Combining the lower and upper bounds, we find that $d(p,q)\leq d_{\tilde{f}_j}(p,q)\leq d(p,q)$ showing $d_{\tilde{f}_j}\rightarrow d$ uniformly.

Now, to show that $d_{\tilde{f}_n}$ does not satisfy a lower Lipshitz bound, consider $p_j=(\frac{1}{2},s_{i,j}-\delta_j)$ and $q_j=(\frac{1}{2},s_{i,j}+\delta_j)$ for any chosen $i=1,...,2^j-1$. Then by calculating the euclidean distance we have 
\begin{align}   
d(p_j,q_j) = |s_{i,j}+\delta_j - (s_{i,j}-\delta_j)| = \left|\frac{i}{2^j} + \frac{1}{2^{2j}}-\left(\frac{i}{2^j} - \frac{1}{2^{2j}}\right)\right| = \frac{2}{2^{2j}}, 
\end{align}
which also implies that $d_{\tilde{f}_n}(p_j,q_j) = \frac{1}{j}\left (\frac{2}{2^{2j}} \right)=\frac{2}{j2^{2j}}$

So if we choose $c_j = \frac{2}{j}$ we can calculate
\begin{align}
  c_jd(p_j,q_j) = \frac{2}{j} \left(\frac{2}{2^{2j}}\right) > \frac{2}{j2^{2j}} = d_{\tilde{f}_n}(p_j,q_j),
\end{align}
and hence we know that $d_{\tilde{f}_n}$ does not satisfy a lower Lipschitz bound. 

\end{proof}

In the next example, we change the rate at which we delete portions of the shortcut to see that we can produce an example which interpolates between Theorem \ref{ex-Shortcut No Uniform No Lipschitz Extreme} and Theorem \ref{ex-Shortcut Dense Uniform no Lipschitz}.

\begin{thm}\label{ex-Shortcut No Uniform No Lipschitz Middle}
For the sequence of functions defined by
 \begin{align}
 S=\left\{s_{i,j}= \tfrac{i}{2^j}\,: \,  i=1,2,... (2^j-1),\, j\in \mathbb{N}\right\}= \left\{ \tfrac{1}{2},\tfrac{1}{4}, \tfrac{2}{4},  \tfrac{3}{4}, 
\tfrac{1}{8},\tfrac{2}{8},\tfrac{3}{8} ,...\right\}
 \end{align}
 which is dense in $[0,1]$
 and
 \begin{align}
 \delta_j:=\left(\frac{1}{2}\right)^{j+2},\quad j \in \mathbb{N}.
 \end{align}
 Let
  \begin{align}
 \bar{f}_j(x,y)=
 \begin{cases}
  \frac{1}{j} & (x,y)\in [\frac{1}{2}-\delta_j,\frac{1}{2}+\delta_j]\times[s_{i,j}-\delta_j, s_{i,j} +\delta_j] \textrm{ for } i =1...2^j-1
 \\ 1 & \textrm{ elsewhere }
 \end{cases}
\end{align}
 we find that $d_{\bar{f}_n}$ converges uniformly to $d_{\bar{f}_{\infty}}$ but does not satisfy a lower Lipschitz bound.
\end{thm}
\begin{proof}
We begin by establishing a useful claim. Let $p=(\frac{1}{2},y_0)$ and $q=(\frac{1}{2},y_1)$ for $y_0,y_1 \in [0,1]$. Let $A_{i,j}=[\frac{1}{2}-\delta_j,\frac{1}{2}+\delta_j]\times[s_{i,j}-\delta_j, s_{i,j} +\delta_j]$ and $A_j = \bigcup^{2^j-1}_{i=1}A_{i,j}$.

Observe that,
\begin{align}
    d_{\bar{f}_j}(p,q) = L(\ell_{pq} \cap A_j^c)+ \frac{1}{j} L(\ell_{pq} \cap A_j)
\end{align}
and let the gap between any two consecutive $A_{i,j}$ be represented by $\eta_j$. The Euclidean distance between the center points of two consecutive $A_{i,j}$ is $|s_{i+1,j}-s_{i,j}|=|\frac{i+1}{2^j} - \frac{i}{2^j}|=|\frac{1}{2^j}|$ for $i =1,2,...,2^j-2$. Since we know the distance between the center point of any $A_{i,j}$ and $s_{i,j}-\delta_j$ is $|(s_{i,j}-\delta_j)-s_{i,j}|=\delta_j$. Hence, the vertical Euclidean distance between any two consecutive $A_{i,j}$ is 
\begin{align}
    \eta_j &= \frac{1}{2^j}-2\delta_j
    =\frac{1}{2^j}-2\left[\left(\frac{1}{2}\right)^{j+2}\right]
= \frac{1}{2^{j+1}}.
\end{align}
We let $N_j = \frac{d(p,q)}{\frac{1}{2^j}}=2^jd(p,q)$ represent an estimate of the number of $A_{i,j}$ between $p$ and $q$. We may observe that $L(\ell_{pq} \cap A^c_j) \le \eta_jN_j$ since we know $\eta_j$ is the maximum distance between two consecutive $A_{i,j}$. Likewise, $L(\ell_{pq}\cap A_j) \le 2\delta_jN_j$ where $2\delta_j$ is the length of each $A_{i,j}$ for any $j \in \N$. Hence, we may see,
\begin{align}
    d_{\bar{f}_j}(p,q) &\le N_j\eta_j+\frac{1}{j}N_j2\delta_j
    \\&=N_j\eta_j + \frac{2\delta_jN_j}{j}
    \\&= \frac{1}{2}d(p,q)+\frac{2[(\frac{1}{2^{j+2}})(2^jd(p,q))]}{j}
    \\&=\frac{1}{2}d(p,q)+\frac{\frac{1}{2}d(p,q)}{j}
    \\&\le\frac{1}{2}d(p,q)+\frac{\sqrt{2}}{2j}
    \\& = d_{\bar{f}_{\infty}}(p,q) + C_j,
\end{align}
where $C_j = \frac{\sqrt{2}}{2j}$. 

Next to show that $d_{\bar{f}_n}$ converges uniformly to $d_{\bar{f}_\infty}$ we will break into cases and the previous claim will be used in the next case.

\textbf{Case 1: }Consider for some $p=(x_0,y_0),q=(x_1,y_1)$, where $p \in [0, \frac{1}{2}) \times [0,1],q \in (\frac{1}{2}, 1] \times [0,1]$. We know that there exists a $J \in \N$ such that for all $j \ge J$ we have $p \notin [\delta_j - \frac{1}{2}, \frac{1}{2}) \times [0,1]$ and $q \notin (\frac{1}{2},\frac{1}{2} + \delta_j] \times[0, 1]$.  Let $p',q' \in M= \{(\frac{1}{2},y) | \ y\in [0,1]\}$. Let $p''\in I_j$ where $I_j = \{\ell_{pp'} \cap \{\frac{1}{2}-\delta_j\} \times [0,1]\}$ and $q'' \in I_j'$ where $I_j'=\{\ell_{q'q} \cap \{\delta_j + \frac{1}{2}\} \times [0,1]\}$. We may observe by Thm. \ref{thm-Distance Lower Bound Estimate},
\begin{align}
    L_{\bar{f}_j}(\ell_{pp'}) \le L(\ell_{pp'}).
\end{align}
Now for $\phi_{pq} = \ell_{pp'} *\ell_{p'q'} * \ell_{q'q}$ we can calculate
\begin{align}
    L_{\bar{f}_j}(\phi_{pq}) &= L_{\bar{f}_j}(\ell_{pp'})+L_{\bar{f}_j}(\ell_{p''p'})+L_{\bar{f}_j}(\ell_{p'q'})+L_{\bar{f}_j}(\ell_{q'q''})+L(\ell_{q''q})
    \\&\le L(\ell_{pp''})+L(\ell_{p''p'})+L_{\bar{f}_j}(\ell_{p'q'})+L(\ell_{q'q''})+L(\ell_{q''q})
    \\&= L(\ell_{pp'})+ L_{\bar{f}_j}(\ell_{p'q'})+L(\ell_{q'q})
    \\&\le d(p,p') + \frac{1}{2}d(p',q') + d(q',q) + C_j,
    \end{align}
    where the last line follows from the first claim.
    Now by taking the infimum and applying Theorem \ref{thm-Explicity Distance Function Half}
    \begin{align}\label{eq-UpperBoundBar}
    d_{\bar{f}_j}(p,q) \le d_{\bar{f}_\infty}(p,q) + C_j.
\end{align}

Let $\gamma$ be a path such that it connects points $p$ and $q$. Denote $r_0=\gamma(t_0)=(x_{r_0},y_{r_0})$ to be the first time it intersects with $(\{\frac{1}{2}\}\times[0,1])$, and $r_1=\gamma(t_1) = (x_{r_1},y_{r_1})$ be the last time it intersects this set. Also let $a_{0,j}=\gamma(s_{0,j})$ be the first time gamma intersects $\{\frac{1}{2}-\delta_j\}\times[0,1]$ for $t\in[0,t_0]$ and $a_{1,j}=\gamma(s_{1,j})$ be the last time $\gamma$ intersects $\{\frac{1}{2}+\delta_j\}\times[0,1]$ for $t\in [t_1,1]$. Calculating the length of $\gamma$ we find 
    \begin{align}
        L_{\bar{f}_j}(\gamma)&= L(\gamma|_{[0,s_{0,j}]})+ L_{\bar{f}_j}(\gamma|_{[s_{0,j},t_0]})+L_{\bar{f}_j}(\gamma|_{[t_0,t_1]})+L_{\bar{f}_j}(\gamma|_{[t_1,s_{1,j}]}) + L(\gamma|_{[s_{1,j},1]})
        \\&\geq L(\ell_{pa_{0,j}})+L_{\bar{f}_j}(\gamma|_{[t_0,t_1]})+L(\ell_{a_{1,j}q})
        \\&= L(\ell_{pr_0}))+(L(\ell_{pa_{0,j}})- L(\ell_{pr_0}))+L_{\bar{f}_j}(\gamma|_{[t_0,t_1]})
        \\&\quad+(L(\ell_{a_{1,j}q})-L(\ell_{r_1q}))+L(\ell_{r_1q})
        \\&= L(\ell_{pr_0})-B_j+L_{\bar{f}_j}(\gamma|_{[t_0,t_1]})-B_j'+L(\ell_{r_1q})
        \\& \geq L(\ell_{pr_0})-B_j+L_{\bar{f}_j}(\gamma|_{[t_0,t_1]})-B_j'+L(\ell_{r_1q}).
    \end{align}
With $L(\ell_{pa_{0,j}})-L(\ell_{pr_0}) = B_j\rightarrow 0$, and $ L(\ell_{a_{1,j}1})-L(\ell_{r_1}q)=B_j'\rightarrow 0$.
 
    Now we want to get a better estimate for $L_{\bar{f}_j}(\gamma|_{[t_0,t_1]})$ . To this end, we claim that we only have to consider paths $\psi_j$ from $r_{0}$ to $r_{1}$ such that $\psi_j\subset [\frac{1}{2}-\delta_j,\frac{1}{2}+\delta_j]\times[0,1]$ $\forall j \in \N$.

To prove this claim, let $\psi_j$ be a path connecting $r_{0}$ and $r_{1}$. Suppose $\psi_j([0,1]) \cap ([\frac{1}{2}-\delta_j,\frac{1}{2}+\delta_j]\times[0,1])^c\not= \emptyset$. Define $E_j=\{e_{0,j},...,e_{m,j}\}$ to be the set of points in which $\psi_j$ exits and enters, consecutively, the set $[\frac{1}{2}-\delta_j,\frac{1}{2}+\delta_j]\times[0,1]$ that lie on its boundary. If we define $c_{i,j}$ to be $\psi (c_{i,j}) = e_{i,j}$ then we can define a path 
\begin{align}
\phi_j(t)=\
    \begin{cases}
        \psi_j(t),& \text{if } t\in [0, c_{i,j}) \cup (c_{i+1,j}, c_{i+2,j}), \text{for } i\in \{0,1,...,m-2\}
        \\\ell_{e_{i,j}e_{i+1,j}},& \text{if } t\in [c_{i,j},c_{i+1,j}], \text{ for } i \in \{0,1,...,m-1\}
    \end{cases}
\end{align}

which is contained in the $[\frac{1}{2}-\delta_j,\frac{1}{2}+\delta_j]\times[0,1]$. By repeated use of the triangle inequality, we then adjust every possible path $\gamma$ of this form and show there exists one with less length.
\begin{align}
    L_{\bar{f}_j}(\psi_j)&=L_{\bar{f}_j}(\psi_j\cap([\frac{1}{2}-\delta_j,\frac{1}{2}+\delta_j]\times[0,1]))+L(\psi_j\cap([\frac{1}{2}-\delta_j,\frac{1}{2}+\delta_j]\times[0,1])^c)
    \\&\geq L_{\bar{f}_j}(\psi_j\cap([\frac{1}{2}-\delta_j,\frac{1}{2}+\delta_j]\times [0,1]) +\sum_{k=1}^m d(e_{k-1,j},e_{k,j}) 
    \\&\geq L_{\bar{f}_j}(\psi_j\cap([\frac{1}{2}-\delta_j,\frac{1}{2}+\delta_j]\times [0,1]) +\sum_{k=1}^mL_{\tilde{f}_j}(\ell_{e_{k-1,j}e_{k,j}})
    \\&=L_{\bar{f}_j}(\phi_j).
\end{align}

Hence, we have established the claim and we only need to consider curves from $r_{0}$ to $r_{1}$ such that $\psi_j\subset [\frac{1}{2}-\delta_j,\frac{1}{2}+\delta_j]\times[0,1]$ $\forall j \in \N$. Furthermore, we may assume that $\psi_j$ is monotone non-decreasing in $y$ since we could always find a monotone curve  which is shorter in length than a curve which is non-monotone in $y$. 

 With this in mind, define another set $U_j=\{u_{k_1,j},u_{k_2,j},...,u_{k_b,j}\}$ to be the consecutive intersections of $\psi_j$ with some $([\frac{1}{2}-\delta_j,\frac{1}{2}+\delta_j]\times \{s_{k,j}-\delta_j,s_{k,j}+\delta_j\}) $. Also, define $v_{k_i,j}$ so that $\gamma(v_{k_i,j})=u_{k_i,j}$ for $i \in \{1,...,b\}$ and $\gamma(t_0)=u_{k_0,j}=\gamma(v_{k_{0,j}})$, $\gamma(t_1)=u_{k_{b+1},j}=\gamma(v_{k_{b+1,j}})$. We then define a path 
\begin{align}
    \tau_j=\ell_{r_0u_{k_1,j}}*...*\ell_{u_{k_b,j}r_1}
\end{align}
and by applying Lemma $\ref{lem-MinimizingCurvesArePiecewiseLinear}$ we know that $\tau_j$ is exactly the type of curve which needs to be considered when estimating the distance. So we can make the observation that $L_{\bar{f}_j}(\psi_j)\geq L_{\bar{f}_j}(\tau_j)$. Let $N_j=\frac{d(r_{0},r_{1})}{\frac{1}{2^j}}=2^jd(r_{0},r_{1})$ be the estimated amount of boxes in between $r_{0}$ and $r_{1}$. Define $A_{i,j} = [\frac{1}{2}-\delta_j, \frac{1}{2} +\delta_j]\times[s_{i,j} -\delta_j,s_{i,j}+\delta_j]$ with $\displaystyle A_j = \bigcup_{0\leq i\leq 2^j-1} A_{i,j}$. Then we can use this to estimate the length of $\tau_j$. 

\begin{align}
    L_{\bar{f}_j}(\tau_j) &= \label{eq-First}L_{\bar{f}_j}(\tau_j \cap A_j) +  L_{\bar{f}_j}(\tau_j\cap (A_j)^c)
    \\ &\geq L(\tau_j\cap (A_j)^c)
    \\&= \frac{1}{2^{j+1}}N_j
    \\&=\frac{1}{2}d(r_{0},r_{1})\label{eq-Last}
\end{align}

We can then use this to bound $L_{\bar{f}_j}(\gamma|_{[t_0,t_1]})$ from below.
    \begin{align}
        L_{\bar{f}_j}(\gamma|_{[t_0,t_1]})&\geq L_{\bar{f}_j}(\tau_j)\label{eq-Question?!}
        \ge \frac{1}{2}d(r_{0},r_{1})
    \end{align}
   Putting together the observations of this case we arrive at a lower bound for the entire path. 
    \begin{align}
        L_{\bar{f}_j}(\gamma_j)&\ge d(p,r_{0})+\frac{1}{2}d(r_{0},r_{1})+d(r_{1},q)-B_j-B_j',
    \end{align}
and by taking the inf over all paths connecting $p$ to $q$ we find
    \begin{align}
       d_{\bar{f}_j}(p,q) &\ge d(p,r_{0})+\frac{1}{2}d(r_{0},r_{1})+d(r_{1},q)  -B_j-B_j'
        \\&\ge d_{\bar{f}_{\infty}}(p,q)  -B_j-B_j'\label{eq-LowerBoundLongArgument2}
    \end{align}
    Finally, by combining with \eqref{eq-UpperBoundBar}, we find
    \begin{align}
        d_{\bar{f}_{\infty}}(p,q)  -B_j-B_j'\le d_{\bar{f}_j}(p,q)  
        \le d_{\bar{f}_{\infty}}(p,q) 
    \end{align}

    \textbf{Case 2:} $p=(\frac{1}{2},y_0)$ and $q=(\frac{1}{2},y_1)$ for $y_0,y_1 \in [0,1]$

    By combining the first claim with \eqref{eq-First}-\eqref{eq-Last} we see that
    \begin{align}
   d_{\bar{f}_\infty}(p,q) - C_j\le  d_{\bar{f}_j}(p,q) \le d_{\bar{f}_\infty}(p,q) + C_j.
\end{align}

    \textbf{Case 3:} Consider for some $p=(x_0,y_0),q=(x_1,y_1)$, where $p,q \in [0, \frac{1}{2}) \times [0,1]$ or $p,q \in (\frac{1}{2}, 1] \times [0,1]$.

    Then we can start by estimating in terms of the straight line connecting $p$ to $q$ and find
    \begin{align}\label{eq-Straight Line Considerations}
        d_{\bar{f}_j}(p,q) & \le L_{\bar{f}_j}(\ell_{pq})\le L(\ell_{pq})=d(p,q).
    \end{align}
    To finish the argument for the upper bound, one can repeat the argument that leads to \eqref{eq-UpperBoundBar}, while including \eqref{eq-Straight Line Considerations} this time, to conclude
    \begin{align}
        d_{\bar{f}_j}(p,q) \le d_{\bar{f}_{\infty}}(p,q) + C_j,
    \end{align}
    where $C_j \searrow 0$ as $j \rightarrow \infty$.
    Similarly, for $j$ chosen large enough we know that $\ell_{pq}$ is not contained in the region $[\frac{1}{2}-\delta_j,\frac{1}{2}+\delta_j]\times [0,1]$, so we can observe that
    \begin{align}\label{eq-Straight Line Lower Bound}
        L_{\bar{f}_j}(\ell_{pq})&= L(\ell_{pq}).
    \end{align}
 Now by repeating almost the same argument that lead to \eqref{eq-LowerBoundLongArgument2}, with curves which attempt to take advantage of the shortcut in the conformal factor around the midline of the square, and combining with \eqref{eq-Straight Line Lower Bound} we find
\begin{align}
        d_{\bar{f}_j}(p,q) \ge d_{\bar{f}_{\infty}}(p,q) - C_j,
    \end{align}
    where $C_j \searrow 0$ as $j \rightarrow \infty$.

    By combining all three cases with Thm \ref{thm-Squeeze} we have $d_{\bar{f}_j} \rightarrow d_{\bar{f}_\infty}$.

    To show that $d_{\bar{f}_\infty}$ does not satisfy a lower Lipshitz bound, consider $p_j=(\frac{1}{2},s_{i,j}-\delta_j)$ and $q_j=(\frac{1}{2},s_{i,j}+\delta_j)$ for any chosen $i=1,...,2^j-1$. Then by calculating distance we have 
\begin{align}   
d(p_j,q_j) = |s_{i,j}+\delta_j - (s_{i,j}-\delta_j)| = \left | 2\delta_j \right| = \frac{2}{2^{j+2}} = \frac{1}{2^{j+1}}
\end{align}
which also implies that $d_{\bar{f}_\infty}(p_j,q_j) =\frac{1}{j} d(p_j, q_j) = \frac{1}{j}\left (\frac{1}{2^{j+1}} \right)=\frac{1}{j2^{j+1}}$

Choose $c_j=\frac{2}{j}$, and by combining previous results we have the following

\begin{align}
    c_j d(p_j,q_j) = \frac{2}{j}\left(\frac{1}{2^{j+1}}\right) = \frac{1}{j2^j} > \frac{1}{j2^{j+1}} =d_{\bar{f}_\infty}(p_j,q_j).
\end{align}
Hence $d_{\bar{f}_\infty}$ does not satisfy a Lipschitz lower bound as desired.
 
\end{proof}

In the last example, we change the rate at which we delete portions of the shortcut again to see that we can produce an example where we pull the shortcut to zero in the limit.

\begin{thm}\label{ex-Shortcut No Uniform No Lipschitz Extreme}
For the sequence of functions defined by
 \begin{align}
 S=\left\{s_{i,j}= \tfrac{i}{2^j}\,: \,  i=1,2,... (2^j-1),\, j\in \mathbb{N}\right\}= \left\{ \tfrac{1}{2},\tfrac{1}{4}, \tfrac{2}{4},  \tfrac{3}{4}, 
\tfrac{1}{8},\tfrac{2}{8},\tfrac{3}{8} ,...\right\}
 \end{align}
 which is dense in $[0,1]$
 and
 \begin{align}
 \delta_j:=\left (\frac{1}{2} \right)^{j+1},\quad j \in \mathbb{N}.
 \end{align}
 Let
  \begin{align}
 \tilde{f}_j(x,y)=
 \begin{cases}
  \frac{1}{j} & (x,y)\in [\frac{1}{2}-\delta_j,\frac{1}{2}+\delta_j]\times[s_{i,j}-\delta_j, s_{i,j} +\delta_j] \textrm{ for } i =1,2, ...,(2^j-1)
 \\ 1 & \textrm{ elsewhere }
 \end{cases}
\end{align}
    we find that $d_{\bar{f}_n}$ converges uniformly to $d_{\hat{f}_\infty}$ but does not satisfy a lower Lipschitz bound.
  
\end{thm}
\begin{proof}
    For the lower bound, notice that $\hat{f}_n \le \tilde{f}_n$ which implies by Theorem \ref{thm-Distance Lower Bound Estimate} that $d_{\hat{f}_n} \le d_{\tilde{f}_n}$  and by Theorem \ref{ex-Shortcut on a Rectangle} we know that $d_{\hat{f}_n} \ge d_{\hat{f}_{\infty}}-C_n$ and hence we see that $d_{\hat{f}_{\infty}}-C_n \le d_{\tilde{f}_n}$.

Now we will establish the upper bound by considering three cases.
    
    \textbf{Case 1: $p \in [0,\frac{1}{2}]\times[0,1]$, $q\in [\frac{1}{2},1]\times[0,1]$}
    
    We know $\min\{|x_0-\frac{1}{2}|+|x_1-\frac{1}{2}|,d(p,q)\} = |x_0-\frac{1}{2}|+|x_1-\frac{1}{2}|$. Let $a_0 = (\frac{1}{2},y_0)$, $a_1 = (\frac{1}{2},y_1)$ and consider the path
   
    \begin{align}
        \alpha = \ell_{p,a_0}*\ell_{a_0,a_1}*\ell_{a_1,q}
    \end{align}
    Let $b_0=(\frac{1}{2},0)$, $b_1=(\frac{1}{2},1)$, $A_{i,j}=(\{\frac{1}{2}\}\times [s_{i,j}-\delta_j])$, and notice that $A_{k,j}\cap A_{l,j}=\emptyset$ for each $k,l\in \{1,2,...,2^j-1\}$, $k \not = l$. To calculate the length under the warping function, we must note that 
    \begin{align}
        L_{\tilde{f}_j}(\ell_{b_0b_1}) &=  L_{\tilde{f}_j}(\ell_{b_0b_1}\cap(\bigcup_{i=1}^{2^j-1} A_{i,j}))+L_{\tilde{f}_j}(\ell_{b_0b_1}\cap(\bigcup_{i=1}^{2^j-1} A_{i,j})^C)
        \\&=   \left(\sum_{i=1}^{2^j-1}  L_{\tilde{f}_j}(\ell_{b_0b_1}\cap A_{i,j})\right)+L_{\tilde{f}_j}(\ell_{b_0b_1}\cap(\bigcup_{i=1}^{2^j-1} A_{i,j})^C)
        \\&= \left(\sum_{i=1}^{2^j-1} \frac{2\delta_j}{j}\right)+\left(1-\left(\sum_{i=1}^{2^j-1} 2\delta_j\right)\right)
        \\& = (2^j-1)\frac{2\delta_j}{j}+(1- (2^j-1)2\delta_j)
    \end{align}
    We can estimate the length of $\alpha$ to be length to be
    \begin{align}
        d_{\tilde{f}_j}(p,q)&\leq L_{\tilde{f}_j}(\alpha)
        \\& =L_{\tilde{f}_j}(\ell_{p,a_0})+L_{\tilde{f}_j}(\ell_{a_0,a_1})+L_{\tilde{f}_j}(\ell_{a_1,q})
        \\&\leq L_{\tilde{f}_j}(\ell_{p,a_0})+L_{\tilde{f}_j}(\ell_{b_0b_1})+L_{\tilde{f}_j}(\ell_{a_1,q})
        \\&\leq|x_0-\frac{1}{2}|+(2^j-1)\frac{2\delta_j}{j}+(1- (2^j-1)2\delta_j)+|\frac{1}{2}-x_1|
        \\&=|x_0-\frac{1}{2}|+(1-\frac{1}{2^{j+1}})\frac{1}{j}+\frac{1}{2^{j+1}}+|\frac{1}{2}-x_1|
        \\&=d_{\hat{f}_\infty}(p,q)+C_n
    \end{align}
    
    \textbf{Case 2: $p,q \in [0,\frac{1}{2}]\times[0,1]$}
    Consider the path $\gamma$ such that it doesn't intersect $\{\frac{1}{2}\}\times[0,1]$, then there exists some $N$ such that $\tilde{f}_j=1$ for all $n\geq N$. In this region, we know the shortest length of any curve connecting $p$ to $q$ will be $d(p,q)$. If $\gamma \cap(\{\frac{1}{2}\}\times[0,1])\not = \emptyset$ then consider the path 
    \begin{align}
        \phi = \ell_{p,a_0}*\ell_{a_0,a_1}*\ell_{a_1,q}
    \end{align}
    We can estimate its length to be
    \begin{align}
        L_{\tilde{f}_j}(\phi)&=L_{\tilde{f}_j}(\ell_{p,a_0})+L_{\tilde{f}_j}(\ell_{a_0,a_1})+L_{\tilde{f}_j}(\ell_{a_1,q})
        \\&\leq L(\ell_{p,a_0})+L_{\tilde{f}_j}(\ell_{(\frac{1}{2},0),(\frac{1}{2},1)})+L(\ell_{a_1,q})
        \\&=|\frac{1}{2}-x_0|+ (1-\frac{1}{2^{j+1}})\frac{1}{j}+\frac{1}{2^{j+1}}+ |\frac{1}{2}-x_1|
        \\&=|\frac{1}{2}-x_0|+|\frac{1}{2}-x_1|+C_j.
    \end{align}
    Since any curve connecting $p$ to $q$ falls into one of these two cases we can take the minimum of these two estimates to find the upper bound
    \begin{align}
        d_{\tilde{f}_j}(p,q)&\leq \min(d(p,q),|\frac{1}{2}-x_0|+|\frac{1}{2}-x_1|+C_j)
        \\&\le \min(d(p,q)+C_j, |\frac{1}{2}-x_0|+|\frac{1}{2}-x_1|+C_j)
        \\&\leq \min(d(p,q), |\frac{1}{2}-x_0|+|\frac{1}{2}-x_1|)+C_j
        \\&=d_{\hat{f}_\infty}(p,q) + C_j.
    \end{align}
    
        \textbf{Case 3: $p,q\in [\frac{1}{2},1]\times[0,1]$}
    Symmetric to Case 2
    
    Combining all three cases and our earlier lower bound estimate, we've shown $d_{\hat{f}_\infty}(p,q)-C_j\leq d_{\tilde{f}_j}(p,q)\leq d_{\hat{f}_\infty}(p,q)+C_j$ proving $d_{\tilde{f}_j}\rightarrow d_{\hat{f}_\infty}$ uniformly.

To show that $d_{\tilde{f_j}}$ does not satisfy lower Lipschitz bound, we will let $p_j = (\frac{1}{2} - \delta_j, s_{i,j})$ and $q_j = (\frac{1}{2} + \delta_j, s_{i,j})$
and the distance between $p_j$ and $d_j$

\begin{align}
 d(p_j,q_j) = \left| \frac{1}{2} - \delta_j - (\frac{1}{2} + \delta_j)\right|
\\
= |-2\delta _j| =  \left(\frac{1}{2^j}\right),
\end{align}
and so then we can calculate the value for $d_{\tilde{f}_j}(p_j,q_j) = \frac{1}{j} d(p_j,q_j) = \frac{1}{j2^{j}}$.

We will pick $c_j = \frac{2}{j}$ to demonstrate that the Lipchitz lower bounds fail since 
\begin{align}
    c_j d(p_j,q_j) = \frac{2}{j} \left( \frac{1}{2^{j}}\right) > \frac{1}{j2^{j}} = d_{\tilde{f}_n}(p_j,q_j).
\end{align}
Hence $d_{\tilde{f_j}}$ fails the Lipschitz lower bound as desired. 
\end{proof}

\section{Proof of Main Theorems}\label{sec-Main Proofs}

We start with the proof of Theorem \ref{thm-Main Theorem Lipschitz Lower Bound} which follows from a simply observation about lengths of curves.

\begin{proof}[Proof of Theorem \ref{thm-Main Theorem Lipschitz Lower Bound}]
    Let $\gamma$ be any piecewise smooth curve connecting $p,q \in [0,1]^2$ and calculate
    \begin{align}
        L_{g_1}(\gamma) &= L_{g_1}(\gamma \cap U)+L_{g_1}(\gamma \cap U^c)
        \\& \ge L_{g_1}(\gamma \cap U)
        \\& \ge c L_{g_0}(\gamma \cap U) \label{eq-UseLowerBoundOnU}
        \\&= c L_{g_0}(\gamma) \label{eq-MeasureZeroCurve}
        \ge c d_{g_0}(p,q),
    \end{align}
    where we used \eqref{eq-LowerBoundOnU} in \eqref{eq-UseLowerBoundOnU} and used that $\mathcal{H}^1_{g_0}(M \setminus U) = 0$ in \eqref{eq-MeasureZeroCurve}. Now by taking the infimum over all curves $\gamma$ we find
\begin{align}
    d_{g_1}(p,q) \ge c d_{g_0}(p,q),
\end{align}
as desired.
\end{proof}

We now establish the proof of Theorem \ref{thm-Main Theorem Uniform Lower Bound} which also follows from a simply observation about lengths of curves.

\begin{proof}[Proof of Theorem \ref{thm-Main Theorem Uniform Lower Bound}]
Let $\gamma$ be any curve piecewise smooth connecting $p,q \in [0,1]^2$ and calculate
    \begin{align}
        L_{g_n}(\gamma) &= L_{g_n}(\gamma \cap U)+L_{g_n}(\gamma \cap U^c)
        \\& \ge L_{g_n}(\gamma \cap U)
        \\& \ge c L_g(\gamma \cap U) \label{eq-UseLowerBoundOnU2}
        \\&\ge c L_g(\gamma) -cC_n \label{eq-DiamToZero}
        \ge c d_g(p,q)-cC_n,
    \end{align}
    where we used \eqref{eq-LowerBoundOnU2} in \eqref{eq-UseLowerBoundOnU2} and used that $\mathcal{H}^1(M \setminus U) \le C_n$ in \eqref{eq-DiamToZero}. Now by taking the infimum over all curves $\gamma$ we find
\begin{align}
    d_{g_n}(p,q) \ge c d_g(p,q)-cC_n,
\end{align}
as desired.
\end{proof}

Moving on to studying upper bounds, we establish Theorem \ref{thm-Main Theorem Lipschitz Upper Bound} which requires foliation of regions by curves. This is a technique which was used extensively in \cite{Allen, Allen-Sormani-2, Allen-Bryden-Second}. The idea is that by foliating a region of full volume by curves connecting points we can relate a volume assumption (or an integral assumption) to a desired conclusion about lengths of curves. Often times this argument involves the coarea formula as we see in the following proof. 

\begin{proof}[Proof of Theorem \ref{thm-Main Theorem Lipschitz Upper Bound}]
 Consider $p,q \in M$ and $\gamma_0$ a distance minimzing smooth geodesic connecting $p$ to $q$ with respect to $g$, i.e. $L_{g}(\gamma_0)= d_g(p,q)$. Now define  $C_{\varepsilon}$ to be a normal neighborhood of $\gamma_0$ of radius $\varepsilon$ which can be foliated by  smooth curves for $\varepsilon>0$ chosen small enough. Then we see that $\vol(U \cap C_{\varepsilon})=\vol(C_{\varepsilon})$ and hence there must be a curve $\gamma$ in the family $C_{\varepsilon}$ so that $L_g(\gamma \cap U) = L_g(\gamma)$. To further illustrate this claim, if every curve $\gamma$ foliating $C_{\varepsilon}$ was such that $L_g(\gamma \cap U) < L_g(\gamma)$ then we would have a contradiction by the coarea formula. Hence, there must be a curve $\gamma$ in the foliation so that $L_g(\gamma \cap U) = L_g(\gamma)$, as desired. Furthermore, this implies that $ L_{g_n}(\gamma) = L_{g_n}(\gamma \cap U)$ since if we parameterize $\gamma$ by $g$ arc length then the $g$ arc length measure, $ds_g$, of $\gamma$ and $\gamma \cap U$ is the same which implies
 \begin{align}
   L_{g_n}(\gamma) &= \int_{\gamma} \sqrt{g_n(\gamma',\gamma')}ds_g  
   = \int_{\gamma\cap U}\sqrt{g_n(\gamma',\gamma')}ds_g =  L_{g_n}(\gamma \cap U). 
 \end{align}
 
Now for any $\delta>0$, by choosing $\varepsilon>0$ small enough, by continuity of the $g$ length of curves foliating $C_{\varepsilon}$ we can ensure that $|L_g(\gamma)-L_g(\gamma_0)|\le \delta$ for all $\gamma$ foliating $C_{\varepsilon}$. This implies that
\begin{align}
    d_{g_n}(p,q) &\le L_{g_n}(\gamma)
    \\& = L_{g_n}(\gamma \cap U)
    \\& \le CL_g(\gamma \cap U)
    \\& = C L_g(\gamma)
    \\& \le CL_g(\gamma_0) + C\delta
    = C d_g(p,q) + C\delta .\label{eq-normal cylinder estimate}
\end{align}
Hence by letting $\delta\rightarrow 0$ we achieve the desired conclusion.
\end{proof}

Lastly, we give the proof of Theorem \ref{thm-Main Theorem Uniform Upper Bound} which follows from using the diameter bound assumption to estimate the lengths of curves from above. 

\begin{proof}[Proof of Theorem \ref{thm-Main Theorem Uniform Upper Bound}]
Consider $p,q \in [0,1]^2$, $\varepsilon>0$, and a piecewise smooth curve so that $L_g(\gamma) \le d(p,q) + \varepsilon$ parameterized with respect to $d$ arclength. Then we can calculate
\begin{align}
    d_{g_n}(p,q) & \le L_{g_n}(\gamma)
    \\&= L_{g_n}\left(\gamma \cap U_n\right)+L_{g_n}\left(\gamma \cap U_n^c\right)
    \\& \le CL_g\left(\gamma \cap U_n\right)+\frac{C_n}{V_n}L_g\left(\gamma \cap U_n^c\right)
    \\& \le CL_g(\gamma )+\frac{C_n}{V_n}\sum _{k=1}^{\infty} \diam(W_k)
    \\& \le C d_g(p,q)+C\varepsilon+C_n.
\end{align}
By taking $\varepsilon \rightarrow 0$ we obtain the desired bound.
\end{proof}

\bibliographystyle{plain}
\bibliography{bibliography}

\end{document}